\documentclass{amsart}

\usepackage{amssymb}
\usepackage{amsmath}
\usepackage{graphicx}
\usepackage{latexsym}
\usepackage{tikz}

\newtheorem{theorem}{Theorem}[section]
\newtheorem{lemma}[theorem]{Lemma}
\newtheorem{proposition}[theorem]{Proposition}
\newtheorem{corollary}[theorem]{Corollary}
\theoremstyle{definition}

\theoremstyle{remark}

\numberwithin{equation}{section}

\def\R{\mathbb{R}}

\def\Rp{\R_+}
\def\Rpn{\Rp^n}
\def\Rpnn{\Rp^{n\times n}}

\def\crit{{\mathcal C}}
\def\Crit{{\mathcal C}}
\def\digr{{\mathcal D}}
\def\Digr{{\mathcal D}}

\def\cT{{\mathcal T}}
\def\bez{\backslash}
\def\proj{{\mathcal P}}
\def\nacht{{\mathcal N}}
\def\ultim{{\mathcal U}}

\def\ulabel{\circ}

\def\modd{\operatorname{mod}}

\def\diag{\operatorname{diag}}
\def\supp{\operatorname{supp}}
\def\begp{\operatorname{beg}}
\def\intp{\operatorname{int}}
\def\endp{\operatorname{end}}
\def\hlabel{h}
\def\ulteq{\overset{T}{=}}
\def\ultleq{\overset{T}{\leq}}
\def\connect{\rightarrow}
\def\strconnect{\rightrightarrows}

\def\Nca{N_c(A)}
\def\Eca{E_c(A)}

\begin{document}

\title{CSR expansions of matrix
powers in max algebra}


\author{Serge{\u{\i}} Sergeev}
\address{University of Birmingham, School of Mathematics,
Watson Building, Edgbaston B15 2TT, UK.}
\email{sergiej@gmail.com}
\thanks{This work was supported
by EPSRC grant RRAH12809 and RFBR grant 08-01-00601.}

\author{Hans Schneider}
\address{Department of Mathematics, University of
Wisconsin-Madison, Madison, Wisconsin 53706, USA.}
\email{hans@math.wisc.edu}


\date{}


\begin{abstract}
We study the behavior of max-algebraic powers of a reducible
nonnegative matrix $A\in\Rpnn$. We show that for $t\geq 3n^2$, the
powers $A^t$ can be expanded in max-algebraic sums of terms of the
form $CS^tR$, where $C$ and $R$ are extracted from columns and rows
of certain Kleene stars, and $S$ is diagonally similar to a Boolean
matrix. We study the properties of individual terms and show that
all terms, for a given $t\geq 3n^2$, can be found in $O(n^4\log n)$
operations. We show that the powers have a well-defined ultimate
behavior, where certain terms are totally or partially suppressed,
thus leading to ultimate $CS^tR$ terms and the corresponding
ultimate expansion. We apply this expansion to the question whether
$\{A^ty,\; t\geq 0\}$ is ultimately linear periodic for each
starting vector $y$, showing that this question can be also answered
in $O(n^4\log n)$ time. We give examples illustrating our main
results.
\end{abstract}

\maketitle

\section{Introduction}

By \textit{max algebra} we understand the analogue of linear algebra
developed over the max-times semiring $\R_{\max,\times}$ which is
the set of nonnegative numbers $\R_+$ equipped with the operations
of ``addition'' $a\oplus b:=\max(a,b)$ and the ordinary
multiplication $a\otimes b:=a\times b$. Zero and unity of this
semiring coincide with the usual $0$ and $1$. The operations of the
semiring are extended to the nonnegative matrices and vectors in the
same way as in conventional linear algebra. That is if $A=(a_{ij})$,
$B=(b_{ij})$ and $C=(c_{ij})$ are matrices of compatible sizes with
entries from $\R_+$,
we write $C=A\oplus B$ if $c_{ij}=a_{ij}\oplus b_{ij}$ for all $i,j$ and $%
C=A\otimes B$ if $c_{ij}=\bigoplus_k a_{ik} b_{kj}=\max_{k}(a_{ik}
b_{kj})$ for all $i,j$.

If $A$ is a square matrix over $\R_+$ then the iterated product
$A\otimes A\otimes
...\otimes A$ in which the symbol $A$ appears $k$ times will be denoted by $%
A^{k}$. These are the {\em max-algebraic powers} of nonnegative
matrices, the main object of our study.

The {\em max-plus semiring} $\R_{\max,+}=(\R\cup\{-\infty\},
\oplus=\max,\otimes=+)$, developed over the set of real numbers $\R$
with adjoined element $-\infty$ and the ordinary addition playing
the role of multiplication, is another isomorphic ``realization'' of
max algebra.  In particular, $x\mapsto\exp(x)$ yields an isomorphism
between $\R_{\max,+}$ and $\R_{\max,\times}$. In the max-plus
setting, the zero element is $-\infty$ and the unity is $0$.

The main results of this paper are formulated in the max-times
setting, since some important connections with nonnegative and
Boolean matrices are more transparent there. However, the max-plus
setting is left for the examples in the last section, in order to
appeal to the readers who work with max-plus algebra and
applications in scheduling problems and discrete event systems
\cite{BCOQ,But:10,CG:79,HOW:05}.

The Cyclicity Theorem is a classical result of max algebra, in the
max-times setting it means that the max-algebraic powers of any
irreducible nonnegative matrix are ultimately periodic (up to a
scalar multiple), with the period equal to cyclicity of the critical
graph. This theorem can be found in Heidergott et al. \cite[Theorem
3.9]{HOW:05}, see also Baccelli et al. \cite[Theorem 3.109]{BCOQ}
and Cuninghame-Green \cite[Theorem 27-6]{CG:79} (all stated in the
max-plus setting). However, the length of the pre-periodic part can
be arbitrarily large and the result does not have an evident
extension to reducible matrices.

The behavior of max-algebraic powers of reducible matrices relies on
connections between their strongly connected components and the
hierarchy of their max-algebraic eigenvalues. This has been well
described in a monograph of Gavalec \cite{Gav:04}, see also
\cite{BdS,Gav-00,Mol-05}.

The relation of the Cyclicity Theorem to the periodicity of Boolean
matrices is understood but not commonly and explicitly used in max
algebra. For instance the construction of cyclic classes
\cite{BV-73,BR} appears in the proof of Lemma 3.3 in \cite{HOW:05}
(without references to literature on Boolean matrices). This
relation becomes particularly apparent after application of a
certain $D^{-1}AD$ similarity scaling called visualization in
\cite{Ser-09,SSB}, see also \cite{ED-99,ED-01}. Semanc\'{\i}kov\'{a}
\cite{Sem-06,Sem-07} realized that the cyclic classes of Boolean
matrices are helpful in treating computational complexity of the
periodicity problems in max-min algebra, with analogous
max-algebraic applications in mind \cite{Blanka:PC}.

A result by Nachtigall \cite{Nacht} states that, though the length
of the preperiodic part cannot be polynomially bounded, the behavior
of matrix powers after $O(n^2)$ can be represented as max-algebraic
sums of matrices from certain periodic sequences. Moln\'{a}rov\'{a}
\cite{Mol-03} studies this {\em Nachtigall expansion} further,
showing that for a given matrix power after $O(n^2)$ the
representing Nachtigall matrices can be computed in $O(n^5)$ time.

In the general reducible case, the sequences of entries
$\{a_{ij}^t,t\geq 0\}$ of $A^t$ are ultimately generalized periodic
\cite{BdS,Mol-05}, meaning that for $t\geq T$ where $T$ is
sufficiently large, they may consist of several ultimately periodic
subsequences which grow with different rates. Gavalec
\cite{Gav-00,Gav:04} showed that deciding whether $\{a_{ij}^t,t\geq
0\}$ is ultimately periodic is in general NP-hard. In a related
study, Butkovi\v{c}~et~al.~\cite{But:10,BCG} considered {\em robust
matrices}, such that for any given $x$ the sequence $\{A^tx\}$ is
ultimately periodic with period $1$. The conditions formulated in
\cite{BCG} can be verified in polynomial time, which suggests that
such ``global'' periodicity questions must be tractable.

The main goal of this paper is to find a common ground for the above
pieces of knowledge on matrix periodicity in max algebra. We
introduce the concept of {\em $CSR$ expansion}, in which a matrix
power $A^t$ is represented as the max-algebraic sum of terms of the
form $CS^tR$, called $CSR$ products. Here $C$ and $R$ have been
extracted from columns and rows of certain Kleene star (the same for
both), and $S$ can be made Boolean by a similarity scaling
$D^{-1}SD$.  The matrix $CR$ appeared previously as spectral
projector \cite{BCOQ, CGG-99}, and $S$ typically arises as the
incidence matrix of a critical graph.

After giving necessary preliminaries in Section \ref{s:prel}, we
start in Section \ref{s:projector} by studying the $CSR$ products.
We show that they form a cyclic group and describe the action of
this group on the underlying Kleene star. We also emphasize the path
sense of these operators, see Theorem~\ref{pathology0}, thus
providing connection to the approach of
\cite{BdS,Gav-00,Gav:04,Mol-05,Sem-06,Sem-07}. In Section
\ref{s:nachtexp} we establish the algebraic form of the Nachtigall
expansion which controls the powers $A^t$ after $t\sim O(n^2)$. See
Theorem~\ref{t:nacht}. In Section \ref{s:ultexp} we show that at
large $t$ certain Nachtigall terms become totally or partially
suppressed by heavier ones, leading to the ultimate expansion of
matrix powers. This result, see Theorem~\ref{t:serg-sch}, can be
understood as a generalization of the Cyclicity Theorem to reducible
case. In Section \ref{s:comp} we treat the computational complexity
of computing the terms of $CSR$ expansions for a given matrix power,
showing that in general this can be done in no more than $O(n^4\log
n)$ operations.
In Section \ref{s:totalper}, which extends the results of
Butkovi\v{c} et al. \cite{But:10,BCG} on robust matrices, we describe {\em
orbit periodic} matrices $A$, i.e., such that the orbit $A^ty$ is
ultimately linear periodic for all initial vectors $y$, see
Theorem~\ref{t:totalper}. We use the ultimate expansion to show that
the conditions for orbit periodicity can be verified in no more than
$O(n^4\log n)$ operations, see Theorem~\ref{t:compcomp-totper} and
its Corollary. We conclude by Section \ref{s:ex} which contains some
examples given in the max-plus setting.

\section{Preliminaries}
\label{s:prel}

In this section we recall some important notions of max algebra.
These are the maximum cycle geometric mean, the critical graph and
the Kleene star. We close the section with nonnegative similarity
scalings and cyclic classes of the critical graph.

Let $A=(a_{ij})\in\Rpnn$. The weighted digraph
$\digr(A)=(N(A),E(A))$, with the set of nodes $N(A)=\{1,\ldots,n\}$
and the set of edges $E(A)=\{(i,j)\mid a_{ij}\neq 0\}$ with weights
$w(i,j)=a_{ij}$, is called the {\em digraph associated} with $A$.
Suppose that $P =(i_{1},...,i_{p})$ is a path in $D_A$, then the \textit{%
weight} of $P$ is defined to be $w(P)=
a_{i_{1}i_{2}}a_{i_{2}i_{3}}\ldots a_{i_{p-1}i_{p}}$ if $p>1$, and
$1$ if $p=1$. If $i_1=i_p$ then $P$ is called a cycle. The {\em length} of $P$, denoted
by $l(P)$, is the number of edges in $P$ (it equals $p-1$ here). 

When any two nodes in $\digr(A)$ can be connected to each other by
paths,
the matrix $A$ is called {\em irreducible}.
Otherwise, it is called {\em reducible}. In the reducible case,
there are some (maximal) strongly connected components of
$\digr(A)$, and a number of nodes that do not belong to any cycle.
We will refer to such nodes as to {\em trivial} components of
$\digr(A)$.

The \textit{maximum cycle geometric mean} of $A$, further denoted by
$\lambda(A)$, is defined by the formula
\begin{equation}
\lambda (A)=\max_{P_c} (w(P_c))^{1/k},  \label{mcm}
\end{equation}%
where the maximization is taken over all cycles
$P_c=(i_1,\ldots,i_k)$, for $k=1,\ldots,n,$ in the digraph
$\digr(A)$.

The Cyclicity Theorem (\cite[Theorem 3.9]{HOW:05}, see also
\cite{BCOQ,But:10,CG:79}) states that if $A$ is irreducible then after a
certain time $T(A)$, there exists a number $\gamma$ such that
$A^{t+\gamma}=\lambda^{\gamma}(A) A^t$ for all $t\geq T(A)$. Thus
$\lambda$ is the ultimate growth rate of matrix powers in this case.
If $\lambda=1$ then $A^{t+\gamma}=A^t$ for $t\geq T(A)$, in which
case we say that $\{A^t,\;t\geq 0\}$ is {\em ultimately periodic}.

Remarkably $\lambda(A)$ is also the largest max-algebraic eigenvalue
of $A$, meaning the largest number $\lambda$ for which there exists
a nonzero $x\in\Rpn$ such that $A\otimes x=\lambda x$, see
\cite{BCOQ,But-03,But:10,CG:79,HOW:05} and references therein.

The operation of taking the maximal cycle geometric mean (m.c.g.m.
for short) is homogeneous: $\lambda(\alpha A)=\alpha\lambda(A)$.
Hence any matrix, which has $\lambda(A)\neq 0$ meaning that
$\digr(A)$ is not acyclic, can be scaled so that
$\lambda(A/\lambda(A))=1$. Following \cite{But-03}, matrix
$A\in\Rpnn$ with $\lambda(A)=1$ will be called {\em definite}.

A cycle $P_c=(i_1,\ldots,i_k)$ in $\digr(A)$ is called {\em
critical}, if $(w(P_c))^{1/k}=\lambda(A)$. Every node and edge that
belongs to a critical cycle is called {\em critical}. The set of
critical nodes is denoted by $\Nca$, the set of critical edges is
denoted by $\Eca$.  The {\em critical digraph} of $A$, further
denoted by $\crit(A)=(\Nca,\Eca)$, is the digraph which consists of
all critical nodes and critical edges of $\digr(A)$.

The cyclicity of an irreducible graph is defined as the g.c.d.
(greatest common divisor) of the lengths of all its simple cycles.
The critical graph defined above may have several strongly connected
components, and in this case the cyclicity is the l.c.m. (least
common multiple) of their cyclicities. This gives the number
$\gamma$ which appears in the Cyclicity Theorem \cite{BCOQ,HOW:05},
and it can be shown that the ultimate period cannot be less than
$\gamma$ (in particular, this follows from the approach of the
present paper).

There is no obvious subtraction in max algebra, however we have an
analogue of $(I-A)^{-1}$ defined by
\begin{equation}
\label{def-kls} A^*:=I\oplus A\oplus A^2\oplus\ldots,
\end{equation}
where $I$ is the identity matrix. This series converges to a finite
matrix if and only if $\lambda(A)\leq 1$
\cite{BCOQ,But-03,CG:79,HOW:05}, and then this matrix
$A^*=(a^*_{ij})$ is called the {\em Kleene star} of $A$. This matrix
has properties $(A^*)^2=A^*$ and, clearly, $A^*\geq I$. It is
important that the entries of max-algebraic powers $A^k=(a^k_{ij})$
express the maximal weights of certain paths: $a^k_{ij}$ is equal to
the greatest weight of paths $P$ that connect $i$ to $j$ and have
length $k$. The entry $a^*_{ij}$ for $i\neq j$ is equal to the
greatest weight of paths that connect $i$ to $j$ with no restriction
on their lengths.

As in the nonnegative linear algebra, we have only few invertible
matrices, in the sense of the existence of (nonnegative) $A^{-1}$
such that $A^{-1}\otimes A=A\otimes A^{-1}=I$. More precisely, such
matrices can be diagonal matrices
\begin{equation}
X=\diag(x):=
\begin{pmatrix}
x_1 &\ldots &0\\
\vdots & \ddots &\vdots\\
0 & \ldots & x_n
\end{pmatrix}
\end{equation}
for a positive $x=(x_1,\ldots,x_n)$, or {\em monomial matrices}
obtained from the diagonal matrices by permutations of rows or
columns. Nevertheless, such matrices give rise to very convenient
{\em diagonal similarity scalings} $A\mapsto X^{-1}AX$. Such
transformations do not change $\lambda(A)$ and $\crit(A)$
\cite{ES-73}. They commute with max-algebraic multiplication of
matrices and hence with the operation of taking the Kleene star.
Geometrically, they correspond to automorphisms of $\Rpn$, both in
the case of max algebra and in the case of nonnegative linear
algebra. The importance of such scalings in max algebra was
emphasized already in \cite{CG:79}, Ch.~28.

By an observation of Fiedler and Pt\'{a}k \cite{FP-67}, for any
definite matrix $A$ there is a scaling $X$ such that $X^{-1}AX$ is
{\em visualized} meaning that all critical entries of the matrix
equal $1$ and all the rest are less than or equal to $1$.  It is
also possible to make $X^{-1}AX$ {\em strictly visualised}
\cite{SSB}, meaning that only critical entries are equal to $1$. If
a matrix is visualised, or strictly visualised, the same is true for
all powers of this matrix, meaning that the critical graph can be
seen as a Boolean matrix that ``lives by itself''. Thus there is a
clear connection to the powers of Boolean matrices.

The periodicity of powers of Boolean matrices is ruled by {\em
cyclic classes} \cite{BV-73} also known as {\em imprimitivity sets}
\cite{BR}, which we explain below. We note that this notion appeared
already in a work of Frobenius \cite{Fro1912}.

\begin{proposition}[e.g. Brualdi-Ryser \cite{BR}]
\label{ryser} Let $G=(N,E)$ be a strongly connected digraph with
cyclicity $\gamma_G$.Then the lengths of any two paths connecting
$i\in N$ to $j\in N$ (with $i,j$ fixed) are congruent modulo
$\gamma_G$.
\end{proposition}

Proposition~\ref{ryser} implies that the following equivalence
relation can be defined: $i\sim j$ if there exists a path $P$ from
$i$ to $j$ such that $l(P)\equiv 0(\modd \gamma_G)$. The equivalence
classes of $G$ with respect to this relation are called {\em cyclic
classes} \cite{BV-73, Sem-06, Sem-07}. The cyclic class of $i$ will
be denoted by $[i]$.

Consider the following {\em access relations} between cyclic
classes: $[i]\to_t[j]$ if there exists a path $P$ from a node in
$[i]$ to a node in $[j]$ such that $l(P)\equiv t(\modd \gamma_G)$.
In this case, a path $P$ with $l(P)\equiv t(\modd \gamma_G)$ exists
between any node in $[i]$ and any node in $[j]$. Further, by
Proposition~\ref{ryser} the length of any path between a node in
$[i]$ and a node in $[j]$ is congruent to $t$, so the relation
$[i]\to_t [j]$ is well-defined.

Cyclic classes can be computed in $O(|E|)$ time by Balcer-Veinott
digraph condensation \cite{BV-73}, where $|E|$ denotes the number of
edges in $G$. At each step of this algorithm, we look for all edges
which issue from a certain node $i$, and condense all end nodes of
these edges into a single node. Another efficient algorithm is
described in \cite{BR}.

Let $S=(s_{ij})$ be the {\em incidence matrix} of $G$, meaning that
$s_{ij}=1$ if $(i,j)\in E$ and $s_{ij}=0$ otherwise.

The ultimate periodicity of such Boolean matrices has been well
studied. If $\gamma_G=1$ then the periodicity of $S^t$ starts latest
after the {\em Wielandt number} $W(n):=(n-1)^2+1$ \cite{BR,Kim:82}.
This bound is sharp and is due to Wielandt \cite{Wie-50}.

If $\gamma>1$ then there are even better sharp bounds due to
Schwartz \cite{Sch-70}. Assume w.l.o.g. that $S$ is irreducible, and
let $n=\alpha\gamma+t$. If $\alpha>1$ then the periodicity starts at
most after $W(\alpha)\gamma+t$ which does not exceed
$\frac{n^2}{\gamma}+\gamma$. If $\alpha=1$ then it starts almost
``straightaway'', after at most $\max(1,t)$.


\section{CSR products}
\label{s:projector}

In this section, given a nonnegative matrix $A\in\Rpnn$, we consider
max-algebraic products of the form $CS^tR$, where $S$ is associated
with some subdigraph of the critical graph $\crit(A)$, and matrices
$C$ and $R$ are extracted from a certain Kleene star related to $A$.
We show that these products form a cyclic group and study their
periodic properties. We also show that $CS^tR$ are related to a
distinguished set of paths which we call $\crit$-heavy.

We start with a remark that the concept of cyclic classes discussed
in Section \ref{s:prel} can be generalized to {\em completely
reducible digraphs}, which consist of (possibly several) strongly
connected components, not connected with each other. Importantly,
the critical digraph of any $A\in\R_+^{n\times n}$ with
$\lambda(A)>0$ is completely reducible.

Let $A\in\Rpnn$ have $\lambda(A)=1$.  Consider any completely
reducible subdigraph $\crit=(N_c,E_c)$ of $\crit(A)$. In particular,
$\crit$ may consist of several disjoint cycles of $\crit(A)$, or we
can take a component of $\crit(A)$, or we may just have
$\crit=\crit(A)$. Denote by $\gamma$ the cyclicity of $\crit$ and
take $B:=(A^{\gamma})^*$. Define the matrices
$C=(c_{ij})\in\Rp^{n\times n}$, $R=(r_{ij})\in\Rp^{n\times n}$ and
$S=(s_{ij})\in\Rp^{n\times n}$ by
\begin{equation}
\label{csrdef}
\begin{split}
c_{ij}&=
\begin{cases}
b_{ij}, &\text{if $j\in N_c$}\\
0, &\text{otherwise,}
\end{cases}\quad
r_{ij}=
\begin{cases}
b_{ij}, &\text{if $i\in N_c$}\\
0, &\text{otherwise,}
\end{cases}\\
s_{ij}&=
\begin{cases}
a_{ij}, &\text{if $(i,j)\in E_c$}\\
0, &\text{otherwise.}
\end{cases}
\end{split}
\end{equation}
The nonzero entries of $C$, respectively $R$, can only be in the
submatrix of $B=(A^{\gamma})^*$ extracted from columns, respectively
rows, in $N_c$. All nonzero entries of $S$ are in the principal
submatrix $S_{N_cN_c}$ extracted from rows and columns in $N_c$.
See Figure~\ref{f:csr} for a schematic display.

\begin{figure}
\centering
\begin{tikzpicture}[scale=0.7]
\draw (0,0) node{$(A^{\gamma})^*$}; \draw (-2.25,2.25)
node{$\crit$}; \draw (-3,3) -- (3,3); \draw (-3,-3) -- (-3,3);
\draw (-3,-3) -- (3,-3); \draw (3,-3) -- (3,3);
\begin{scope}[dashed]
\draw (-2.25,2.25) circle (0.75cm);
\end{scope}
\begin{scope}[red]
\draw (3,3) -- (-3,3); \draw (3,2.5) -- (-3,2.5); \draw (3,2) --
(-3,2); \draw (3,1.5) -- (-3,1.5); \draw (-3.3,2.25) node{$R$};
\end{scope}
\begin{scope}[brown]
\draw (-3,3) -- (-3,-3); \draw (-2.5,3) -- (-2.5,-3); \draw (-2,3)
-- (-2,-3); \draw (-1.5,3) -- (-1.5,-3); \draw (-2.25,3.3)
node{$C$};
\end{scope}
\end{tikzpicture}
\caption{The scheme of $C$ and $R$ defined in~\eqref{csrdef}}
\label{f:csr}
\end{figure}
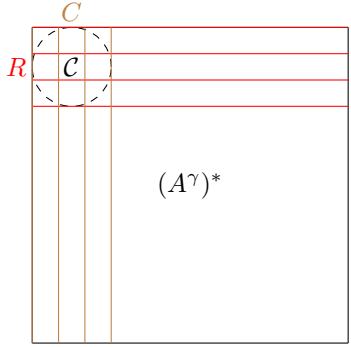

We show in the next proposition that $S$ can be assumed to be $0-1$.
It can be also deduced from the results in \cite{ES-73}.

\begin{proposition}
\label{p:s0-1} Let $A\in\Rpnn$ have $\lambda(A)=1$ and let $S$ be
defined as above. There exists a positive $z\in\Rp^n$ such that
$D^{-1}SD$ for $D:=\diag(z)$ is a $0-1$ matrix.
\end{proposition}
\begin{proof}
As $\lambda(S)=1$, we can take $z:=\bigoplus_{j=1}^n S_{\cdot j}^*$.
This vector is positive, and observe that $Sz\leq z$ since $SS^*\leq
S^*$. From this and $\digr(S)=\crit(S)=\crit$, it can be deduced by
multiplying $z_i^{-1}s_{ij}z_j\leq 1$ along cycles that
$z_i^{-1}s_{ij}z_j=1$ for all $(i,j)\in \digr(S)$, while
$z_i^{-1}s_{ij}z_j=0$ for all $(i,j)\notin\digr(S)$.
\end{proof}

As $S$ can be scaled to $0-1$ matrix, we conclude that
$\{S^t,\;t\geq 0\}$ becomes periodic at most after the Wielandt
number $W(n_c)=(n_c-1)^2+1$, where $n_c$ is the number of nodes in
$N_c$. Matrix $A$ will be called {\em $S$-visualized}, if $S$
defined in \eqref{csrdef} is Boolean. In the $S$-visualized case,
the asymptotic form of $S^t$ for $t\geq T$ is determined by the
cyclic classes of $\crit$ \cite{BR}:
\begin{equation}
\label{e:sult}
s_{ij}^t=
\begin{cases}
1, &\text{if $[i]\to_t[j]$,}\\
0, &\text{otherwise.}
\end{cases}
\end{equation}

Now we study the {\em $CSR$ products}
\begin{equation}
\label{csr-def} \proj^{(t)}:=CS^tR,
\end{equation}
assuming that $A\in\Rpnn$ has $\lambda(A)=1$.

We start by the observation that if $T$ is the number after which
$\{S^t,\;t\geq 0\}$ becomes periodic, then
\begin{equation}
\label{e:sr=r} S^{l\gamma}R=R,\quad CS^{l\gamma}=C,\quad\forall
l\gamma\geq T.
\end{equation}
Indeed, \eqref{e:sult} implies that all diagonal entries of
$S^{l\gamma}$ with indices in $N_c$ are $1$ if $l\gamma\geq T$,
which implies $S^{l\gamma}R\geq R$ and $CS^{l\gamma}\geq C$. On the
other hand, $S\leq A$ and so $S^{l\gamma}\leq (A^{\gamma})^*$.
Further, $(A^{\gamma})^*R\leq R$ and $C(A^{\gamma})^*\leq C$ since
$(A^{\gamma})^*(A^{\gamma})^*=(A^{\gamma})^*$, and so
$S^{l\gamma}R\leq(A^{\gamma})^*R\leq R$ and $CS^{l\gamma}\leq
C(A^{\gamma})^*\leq C$.

As $\{S^t,\;t\geq T\}$ is periodic, so is $\{\proj^{(t)},\;t\geq
T\}$. Moreover, we conclude from \eqref{e:sr=r} that this
periodicity starts from the very beginning.

\begin{proposition}[Periodicity]
\label{p:period} $\proj^{(t+\gamma)}=\proj^{(t)}$ for all $t\geq 0$.
\end{proposition}
\begin{proof}
It follows from Eqn. \eqref{e:sr=r} that
$\proj^{(t+l\gamma)}=\proj^{(t)}$ and
$\proj^{(t+(l+1)\gamma)}=\proj^{(t+\gamma)}$ for $l\gamma\geq T$.
But $\proj^{(t+l\gamma)}=\proj^{(t+(l+1)\gamma)}$, which implies
that also $\proj^{(t+\gamma)}=\proj^{(t)}$.
\end{proof}

It is also useful to understand the meaning of $\proj^{(t)}$ in
terms of paths. Given a set of paths $\Pi$, we denote by $w(\Pi)$
the greatest weight of paths in $\Pi$, assuming $w(\Pi)=0$ if $\Pi$
is empty. A path will be called {\em $\crit$-heavy} if it goes
through a node in $N_c$. The set of all $\Crit$-heavy paths on
$\digr(A)$ that connect $i$ to $j$ and have length $t$ will be
denoted by $\Pi_{ij,t}^{\hlabel}$. We also denote by $\tau$ the
maximal cyclicity of the components of $\crit(A)$.

\begin{theorem}[$\Crit$-Heavy Paths]
\label{pathology0} Let $A\in\Rpnn$ have $\lambda(A)=1$ and let
$T\geq 0$ be such that $\{S^t,t\geq T\}$ is periodic.
\begin{itemize}
\item[1.] For $t\geq 0$,
\begin{equation}
\label{pathsleq}
w(\Pi_{ij,t}^{\hlabel})\leq\proj^{(t)}_{ij}.
\end{equation}
\item[2.] For $t\geq T+2\tau(n-1)$,
\begin{equation}
\label{pathsgeq}
w(\Pi_{ij,t}^{\hlabel})\geq\proj^{(t)}_{ij}.
\end{equation}
\end{itemize}
\end{theorem}
\begin{proof} By Proposition~\ref{p:s0-1} there exists a diagonal matrix
$D$ such that $D^{-1}AD$ is $S$-visualized. As both sides of
\eqref{pathsleq} and \eqref{pathsgeq} are stable under similarity
scaling of $A$, we will assume that $A$ is already $S$-visualized.

1.: Let $P\in\Pi_{ij,t}^{\hlabel}$ and $w(P)\neq 0$. We need to show
that
\begin{equation}
\label{wpleq}
w(P)\leq \proj^{(t)}_{ij}
\end{equation}
Path $P$ can be decomposed as $P=P_{\begp}\circ P_{\endp}$, where
$P_{\begp}$ connects $i$ to a node $m\in N_c$ and $P_{\endp}$
connects $m$ to $j$. Adjoining to $P$ any sufficiently large number
of cycles of $\crit$ that go through $m$ and whose total length is a
multiple of $\gamma$, we obtain a path $P'$ that we can decompose as
$P'=P'_{\begp}\circ P'_{\intp}\circ P'_{\endp}$, where $P'_{\begp}$
connects $i$ to $m_1\in\crit$ and $l(P'_{\begp})$ is a multiple of
$\gamma$, $P'_{\intp}$ connects $m_1$ to $m_2\in N_c$, has length
$t$ and belongs entirely to $\crit$, and $P'_{\endp}$ connects $m_2$
to $j$ and $l(P'_{\endp})$ is a multiple of $\gamma$. We conclude
that $w(P'_{\begp})\leq c_{im_1}$, $w(P'_{\endp})\leq r_{m_2j}$ and
$w(P'_{\intp})=s_{m_1m_2}^t=1$. We obtain $w(P)=w(P')\leq
c_{im_1}s_{m_1m_2}^t r_{m_2j}$, which implies \eqref{wpleq}, and
hence \eqref{pathsleq}.

2.: 
There exist indices $m_1$ and $m_2$ such that
$\proj_{ij}^{(t)}=c_{im_1}s_{m_1m_2}^tr_{m_2j}$. This is the
weight of a path $P$ decomposed as $P=P_{\begp}\circ P_{\intp}\circ
P_{\endp}$, where $P_{\begp}$ connects $i$ to $m_1$, $P_{\intp}$
connects $m_1$ to $m_2$ and $P_{\endp}$ connects $m_2$ to $j$. Here
$P_{\intp}$ has length $t$ and belongs to the component of $\crit$
which we denote by $\cT$ and whose cyclicity we denote by $\pi$. The
lengths of $P_{\begp}$ and $P_{\endp}$ are respectively
$l(P_{\begp})=l_1\pi$ and $l(P_{\endp})=l_2\pi$ for some $l_1,l_2$,
since $l(P_{\begp})$ and $l(P_{\endp})$ are multiples of $\gamma$
and  $\gamma$ is itself a multiple of $\pi$. Paths $P_{\begp}$ and
$P_{\endp}$ correspond to certain paths on $\digr(A^{\pi})$ with
lengths $l_1$ and $l_2$. If $l_1\geq n$ or $l_2\geq
n$, then we can perform cycle deletion (w.r.t.
$\digr(A^{\pi})$) and obtain paths $P^2_{\begp}$ and $P^2_{\endp}$
with lengths $k_1\pi$ and $k_2\pi$ where $k_1<n$ and $k_2<n$. For
the resulting path $P^2:=P^2_{\begp}\circ P_{\intp}\circ
P^2_{\endp}$ we will have $w(P^2)\geq w(P)$ since $\lambda=1$. Now
we have $l(P^2_{\begp})+l(P^2_{\endp})\leq 2\tau(n-1)$. If $t\geq
T+2\tau(n-1)$, then the principal submatrix of $S^t$ corresponding
to the component $\cT$ coincides with that of $S^{t-(k_1+k_2)\pi}$,
which implies that $P_{\intp}$ can be replaced by a path
$P_{\intp}^3$ with length $t-(k_1+k_2)\pi$, so that
$w(P^3)=w(P^2)\geq \proj_{ij}^{(t)}$ where $P^3:=P_{\begp}^2\circ
P_{\intp}^3\circ P_{\endp}^2\in\Pi_{ij,t}^{\hlabel}$.
\end{proof}

This ``path sense'' of $\proj^{(t)}$ simplifies the proof of the
following important law.

\begin{theorem}[Group law]
\label{t:group} $\proj^{(t_1+t_2)}=\proj^{(t_1)}\proj^{(t_2)}$ for
all $t_1,t_2\geq 0$
\end{theorem}
\begin{proof}
As $(A^{\gamma})^*(A^{\gamma})^*=(A^{\gamma})^*$, we have
$(RC)_{ii}=(A^{\gamma})^*_{ii}=1$ for all $i\in N_c$, and hence
$RCS\geq S$. We use this to obtain that
$\proj^{(t_1)}\proj^{(t_2)}\geq\proj^{(t_1+t_2)}$:
\begin{equation}
\label{ineq1111} \proj^{(t_1)}\proj^{(t_2)}=CS^{t_1}RCS^{t_2}R\geq
CS^{t_1+t_2}R=\proj^{(t_1+t_2)}.
\end{equation}
But Theorem \ref{pathology0} implies that
$\proj^{(t_1)}\proj^{(t_2)}\leq\proj^{(t_1+t_2)}$ for all large
enough $t_1,t_2$, since the concatenation of two $\crit$-heavy paths
is again a $\crit$-heavy path. As $\proj^{(t)}$ are periodic, it
follows that $\proj^{(t_1)}\proj^{(t_2)}\leq\proj^{(t_1+t_2)}$ for
all $t_1,t_2$, and we obtain the claim combining this with the
reverse inequality.
\end{proof}

Formulas \eqref{e:sult} and \eqref{e:sr=r} imply that if $A$ is
$S$-visualized then all rows of $R$ or columns of $C$ with indices
in the same cyclic class of $\crit$ coincide.
Hence, when working with $\proj^{(t)}$
we can assume without loss of generality that all cyclic classes
have just $1$ element and consequently, that $S_{N_cN_c}$ is a
permutation matrix. This captures the structure of $\proj^{(t)}$,
which form a cyclic group of order $\gamma$.

As usual $e_i$ denotes the vector which has all coordinates equal to
$0$ except for the $i$th which equals $1$. For the rows
$\proj_{i\cdot}^{(t)}$ and columns $\proj_{\cdot j}^{(t)}$ of
$\proj^{(t)}$ we have:
\begin{equation}
e_i^T\proj^{(t)}=\proj^{(t)}_{i\cdot},\
\proj^{(t)}e_j=\proj^{(t)}_{\cdot j}.
\end{equation}

Next we study the periodicity of $\proj^{(t)}$ in more detail. It
turns out that the columns and rows of $\proj^{(t+1)}$ with indices
in $N_c$ can be obtained from those of $\proj^{(t)}$ by means of a
permutation on cyclic classes, while the rest of the columns (or
rows) are max-linear combinations of the critical ones. We start
with the following observation on the spectral projector
$\proj^{(0)}:=CR$, which can be found in \cite{BCOQ,CGG-99}.

\begin{lemma}
\label{l:sp} $\proj_{i\cdot}^{(0)}=R_{i\cdot}$ and $\proj_{\cdot
i}^{(0)}=C_{\cdot i}$ for all $i\in N_c$.
\end{lemma}
\begin{proof}
As $(A^{\gamma})^*_{ii}=1$ for all $i\in N_c$, we obtain $e_i^T
C=C_{i\cdot}\geq e_i^T$ and $Re_i=R_{\cdot i}\geq e_i$. We see that
\begin{equation}
\begin{split}
\proj_{i\cdot}^{(0)}&=e_i^TCR\geq e_i^TR=R_{i\cdot},\\
\proj_{\cdot i}^{(0)}&=CRe_i\geq Ce_i=C_{\cdot i}.
\end{split}
\end{equation}
But $CR\leq ((A^{\gamma})^*)^2=(A^{\gamma})^*$, which implies the
reverse inequalities $\proj_{i\cdot}^{(0)}\leq R_{i\cdot}$ and
$\proj_{\cdot i}^{(0)}\leq C_{\cdot i}$.
\end{proof}

\begin{theorem}
\label{t:rotate} Let $A\in\Rpnn$ have $\lambda(A)=1$ and be
$S$-visualized. If $[i]\to_t [j],$ then
\begin{equation}
\label{e:rotate} \proj_{i\cdot}^{(t+s)}=\proj_{j\cdot}^{(s)},\
\proj_{\cdot i}^{(s)}=\proj_{\cdot j}^{(t+s)}
\end{equation}
for all $s,t\geq 0$.
\end{theorem}
\begin{proof}
We prove the first equality of \eqref{e:rotate}. Using the group law
we assume that $s=0$.  We also assume that $S_{N_cN_c}$ is a
permutation matrix, then $e_i^TS^t=e_j$. Using this and
Lemma~\ref{l:sp} we obtain
\begin{equation}
\label{e:ineq11}
\begin{split}
\proj_{j\cdot}^{(0)}=e_j^T\proj^{(0)}&=e_j^TR=e_i^TS^tR\leq\\
&\leq e_i^TCS^tR=e_i^T\proj^{(t)}=\proj^t_{i\cdot}.
\end{split}
\end{equation}
Analogously we have $\proj_{i\cdot}^{(0)}\leq
\proj_{j\cdot}^{(\gamma-t)}$. Multiplying this inequality by
$\proj^{(t)}$ and using the group law and periodicity, we obtain
that
$\proj_{i\cdot}^{(t)}\leq\proj_{j\cdot}^{(\gamma)}=\proj_{j\cdot}^{(0)}$.
Combining this with \eqref{e:ineq11} we obtain the desired property.
\end{proof}

\begin{corollary}
\label{c:per} Let $A\in\Rpnn$ have $\lambda(A)=1$. Then
$\proj_{i\cdot}^{(t)}=(S^tR)_{i\cdot}$ and $\proj_{\cdot i}^{(t)}=
(CS^t)_{\cdot i}$ for all $i\in N_c$.
\end{corollary}
\begin{proof}
We assume that $A$ is $S$-visualized, and we also assume that
$S_{N_cN_c}$ is a permutation matrix. For $[i]\to_t [j]$, Theorem
\ref{t:rotate} and Lemma \ref{l:sp} imply that
\begin{equation*}
\begin{split}
\proj_{i\cdot}^{(t)}&=e_i^T\proj^{(t)}=e_j^T\proj^{(0)}=e_j^TR=
e_i^TS^tR=(S^tR)_{i\cdot},\\
\proj_{\cdot j}^{(t)}&=\proj^{(t)}e_j=\proj^{(0)}e_i=Ce_i=CS^te_j=
(CS^t)_{\cdot j}.
\end{split}
\end{equation*}
The claim is proved.
\end{proof}

\begin{corollary}
\label{c:lindep} Let $A\in\Rpnn$ have $\lambda(A)=1$. For each
$k=1,\ldots,n$ there exist $\alpha_{ik}$ and $\beta_{ki}$, where
$k\in N_c$, such that
\begin{equation}
\label{e:lincomb} \proj_{\cdot k}^{(t)}=\bigoplus_{i\in N_c}
\alpha_{ik}\proj_{\cdot i}^{(t)},\quad
\proj_{k\cdot}^{(t)}=\bigoplus_{i\in N_c}
\beta_{ki}\proj_{i\cdot}^{(t)}.
\end{equation}
\end{corollary}
\begin{proof}
By Corollary \ref{c:per} we have $\proj_{\cdot
i}^{(t)}=(CS^t)_{\cdot i}$ and
$\proj_{i\cdot}^{(t)}=(S^tR)_{i\cdot}$ for all $i\in N_c$. Eqn.
\eqref{e:lincomb} follows directly from $\proj^{(t)}=CS^tR$, the
coefficients $\alpha_{ik}$ (resp. $\beta_{ki}$) being taken from the
$k$th column of $R$ (resp. the $k$th row of $C$).
\end{proof}

\section{Nachtigall expansions}
\label{s:nachtexp} In this section we show that the powers $A^t$ of
$A\in\Rpnn$ can be expanded for $t\geq 3n^2$ as sum of $CSR$
products. This establishes a more general algebraic form of the
Nachtigall expansion studied in \cite{Mol-03,Nacht}.

Let $A=(a_{ij})\in\Rpnn$. Define $\lambda_1=\lambda(A)$, and let
$\crit_1=(N_1,E_1)$ be a completely reducible subdigraph of
$\crit(A)$. Set $A_1:=A$ and $K_1:=N$, where $N=\{1,\ldots,n\}$.

The elements of a Nachtigall expansion will be now defined
inductively for $\mu\geq 2$. Namely, we define $K_{\mu}:=N\bez
\cup_{i=1}^{\mu-1} N_i$ and $A_{\mu}=(a^{\mu}_{ij})\in\Rpnn$ by

\begin{equation}
\label{amu} a^{\mu}_{ij}=
\begin{cases}
a_{ij}, & \text{if $i,j\in K_{\mu}$,}\\
0, & \text{otherwise.}
\end{cases}
\end{equation}

Further define $\lambda_{\mu}:=\lambda(A_{\mu})$. If
$\lambda_{\mu}=0$ then stop, otherwise select a completely reducible
subdigraph $\crit_{\mu}=(N_{\mu},E_{\mu})$ of the critical digraph
$\crit(A_{\mu})$, and proceed as above with $\mu:=\mu+1$.

By the above procedure we define $K_{\mu}$, $A_{\mu}$,
$\lambda_{\mu}$ and $\crit_{\mu}=(N_{\mu},E_{\mu})$ for
$\mu=1,\ldots,m$, where $m\leq n$ is the last number $\mu$ such that
$\lambda_{\mu}>0$.

Denote $L:=\bigcup_{i=1}^m N_{\mu}$ and $\overline{L}=N\bez L$.

For each $\mu=1,\ldots,m$, let $\gamma_{\mu}$ be the cyclicity of
$\crit_{\mu}$. Since
$\lambda((A_{\mu}/\lambda_{\mu})^{\gamma_{\mu}})
=\lambda(A_{\mu}/\lambda_{\mu})=1$, the Kleene star
$B_{\mu}:=((A_{\mu}/\lambda_{\mu})^{\gamma_{\mu}})^*$ is finite.
Define the matrices $C_{\mu}=(c_{ij}^{\mu})\in\Rpnn$,
$R_{\mu}=(r_{ij}^{\mu})\in\Rpnn$ and
$S_{\mu}=(s_{ij}^{\mu})\in\Rpnn$ by
\begin{equation}
\label{CRSdef}
\begin{split}
c_{ij}^{\mu}&=
\begin{cases}
b_{ij}^{\mu}, &\text{if $j\in N_{\mu}$},\\
0, &\text{otherwise,}
\end{cases}\quad
r_{ij}^{\mu}=
\begin{cases}
b_{ij}^{\mu}, &\text{if $i\in N_{\mu}$},\\
0, &\text{otherwise,}
\end{cases}\\
s_{ij}^{\mu}&=
\begin{cases}
a_{ij}^{\mu}/\lambda_{\mu}, &\text{if $(i,j)\in E_{\mu}$}\\
0, &\text{otherwise.}
\end{cases}
\end{split}
\end{equation}

A schematic example of Nachtigall expansion is given in Figure~\ref{f:nacht}.

\begin{figure}
\begin{tabular}{cccc}
\begin{tikzpicture}[scale=0.7]
\draw (-2.5,2.5) node[circle,draw]{$S_1$}; 
\begin{scope}[red]
\foreach \i in
{-1.5,-0.5, 0.5, 1.5,2.5} {
\draw (\i ,2.5) node{$*$};
}
\end{scope}

\begin{scope}[brown]
\foreach \i in {1.5,0.5}
{
\draw (-2.5,\i) node{$*$};
}

\foreach \i in {-0.5, -1.5,-2.5} { \draw (-2.5,\i) node{$0$};
}
\end{scope}

\foreach \i in {0.5,1.5,2.5}
{
\foreach \j in {1.5,0.5,-0.5,-1.5,-2.5}
{
\draw (\i,\j) node{$*$};
}
}

\foreach \i in {-1.5,-0.5}
{
\foreach \j in {1.5,0.5}
{
\draw (\i,\j) node{$*$};
}
}

\foreach \i in {-1.5,-0.5} 
{ 
\foreach \j in {-0.5,-1.5,-2.5} 
{ 
\draw(\i,\j) node{$0$}; } }

\draw (-3,3) -- (3,3);
\draw (-3,-3) -- (-3,3); \draw (-3,-3) -- (3,-3); \draw (3,-3) --
(3,3);

\begin{scope}[dashed,brown]
\draw (-2,3) -- (-2,-3);
\draw (-2.5,-3.3) node{$C_1$};
\end{scope}

\begin{scope}[dashed,red]
\draw (-3,2) -- (3,2);
\draw (3.4,2.5) node{$R_1$};
\end{scope}

\if{
\begin{scope}[dashed]
\foreach \j in {2.25,2.5,2.75}
{
\draw (-3,\j) -- (3,\j);
}
}\fi


\if{
\begin{scope}[dashed]
\foreach \i in {-2.25,-2.5,-2.75}
{
\draw (\i,-3) -- (\i,3);
}
}\fi


\if{
\begin{scope}[shorten >=1pt,->]
\draw (2,2.5) -- (4.5,2.5);
\end{scope}
\draw  (4.5,2.5) node[anchor=west]{$\lambda_1, C_1, S_1, R_1$};
}\fi

\end{tikzpicture}

&&

\begin{tikzpicture}[scale=0.7]
\draw (-2.5,2.5) node{$0$}; \foreach \i in {-1.5,-0.5, 0.5,
1.5,2.5} {
\draw (\i ,2.5) node{$0$}; }

\foreach \i in {1.5,0.5, -0.5, -1.5,-2.5} { \draw (-2.5,\i)
node{$0$};
}

\begin{scope}[red]
\foreach \i in {-0.5,0.5,1.5,2.5}
{
\draw (\i ,1.5) node{$*$};
}
\end{scope}

\begin{scope}[brown]
\foreach \i in {0.5} { \draw (-1.5,\i) node{$*$};
\foreach \i in {-0.5, -1.5,-2.5} 
{ \draw (-1.5,\i) node{$0$};}
}
\end{scope}

\foreach \i in {0.5} 
{ \draw (-0.5,\i) node{$*$};}

\foreach \i in {-0.5, -1.5,-2.5} { \draw (-0.5,\i) node{$0$};
}

\foreach \i in {0.5,1.5,2.5} { \foreach \j in {0.5,-0.5,-1.5,-2.5} {
\draw (\i,\j) node{$*$}; } } \draw (-1.5,1.5)
node[circle,draw]{$S_2$}; \draw (-3,3) -- (3,3); \draw (-3,-3)
-- (-3,3); \draw (-3,-3) -- (3,-3); \draw (3,-3) -- (3,3);

\begin{scope}[dashed,brown]
\draw (-1,2) -- (-1,-3);
\draw (-1.5,-3.3) node{$C_2$};
\end{scope}

\begin{scope}[dashed,red]
\draw (-2,1) -- (3,1);
\draw (3.4,1.5) node{$R_2$};
\end{scope}

\draw (-2,2) -- (-2,-3);
\draw (-2,2) -- (3,2);

\if{
\begin{scope}[shorten >=1pt,->]
\draw (2,1.5) -- (4.5,1.5);
\end{scope}
\draw  (4.5,1.5) node[anchor=west]{$\lambda_2, C_2, S_2, R_2$};
}\fi
\end{tikzpicture}
\end{tabular}
\caption{Formation of first two terms in a 
Nachtigall expansion: a schematic example. Note that the structure of strongly connected
components of $\digr(A)$ is of no use here.}
\label{f:nacht}
\end{figure}
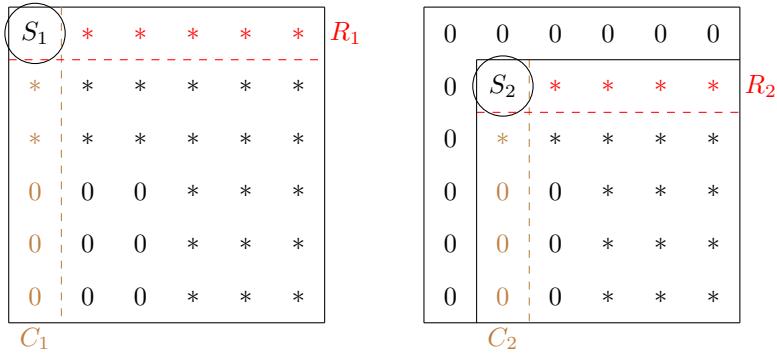



It follows from Proposition~\ref{p:s0-1} that each $S_{\mu}$ can be
scaled to a $0-1$ matrix using a certain vector denoted here by
$z^{\mu}$. Note that the sets $N_{\mu}$ are pairwise disjoint.
Defining $z\in\Rpn$ by
\begin{equation}
\label{fdef} z_i=
\begin{cases}
z^{\mu}_i, & \text{if $i\in N_{\mu}$},\\
1, & \text{if $i\in\overline{L}$},
\end{cases}
\end{equation}
and letting $D:=\diag(z)$, we obtain that the matrix
$\Tilde{A}:=D^{-1}AD$ is {\em totally $S$-visualized}, meaning that
all corresponding matrices $\Tilde{S}_{\mu}$ are Boolean.

As $S_{\mu}$ can be scaled to be Boolean, the sequences of their
max-algebraic powers $\{S_{\mu}^t\mid t\geq 0\}$, being powers of
Boolean matrices when scaled, are ultimately periodic with periods
$\gamma_{\mu}$.  This periodicity starts at most after the
corresponding Wielandt numbers $(k_{\mu}-1)^2+1$ where $k_{\mu}$ is
the number of elements in $N_{\mu}$

We proceed with some notation. Denote  $\mu(i)=\mu$ if $i\in
N_{\mu}$, and $\mu(i)=+\infty$ if $i\in \overline{L}$. Denote by
$\Pi_{ij,t}$ the set of paths on $\digr(A)$ which connect $i$ to $j$
and have length $t$. Denote by $\Pi_{ij,t}^{\mu}$ the set of paths
$P\in\Pi_{ij,t}$ such that $\min_{i\in N_P} \mu(i)=\mu$, where $N_P$
is the set of nodes visited by $P$. (Note that greater values of
$\mu$ correspond to smaller $\lambda_{\mu}$.) The paths in
$\Pi_{ij,t}^{\mu}$ will be called {\em $\mu$-heavy}, since they are
$\crit_{\mu}$-heavy (see Sect. \ref{s:projector}) in $A_{\mu}$.

Any path $P\in\Pi_{ij,t}$ with $t\geq n$ has at least one cycle.
Hence there are no paths with length $t\geq n$ that visit only the
nodes in $\overline{L}$, for otherwise the subdigraph of $\Digr(A)$
induced by $\overline{L}$ would contain a cycle and the number of
components would be more than $m$. We can express the entries of
$A^t=(a^t_{ij})$ and $A_{\mu}^t=(a^{\mu,t}_{ij})$ for $t\geq n$ as
follows:
\begin{equation}
\label{aijmut}
\begin{split}
a^t_{ij}&=w(\Pi_{ij,t})=\bigoplus_{\mu=1}^m w(\Pi_{ij,t}^{\mu}),\quad t\geq n\\
a_{ij}^{\mu,t}&=\bigoplus_{\nu=\mu}^m w(\Pi_{ij,t}^{\nu}),\quad
t\geq n.
\end{split}
\end{equation}

Denote
\begin{equation}
\nacht_{\mu}^{(t)}:=C_{\mu}S^t_{\mu}R_{\mu}.
\end{equation}
These $CSR$ products are defined from $A_{\mu}$ and $\crit_{\mu}$ in
the same way as $\proj^{(t)}$, see Sect. \ref{s:projector}, were
defined from $A$ and $\crit$, and it follows that the sequence
$\nacht_{\mu}^{(t)}$ is periodic with period $\gamma_{\mu}$. Further
denote by $\tau_{\mu}$ the greatest cyclicity of a component in
$\crit_{\mu}$ and by $n_{\mu}$ the number of nodes in $K_{\mu}$. The
following is a version of Theorem~\ref{pathology0} for
$\nacht_{\mu}^{(t)}$.

\begin{theorem}[$\mu$-Heavy Paths]
\label{pathology1} Let $T_{\mu}$ be such that $\{S_{\mu}^t,\;t\geq
T_{\mu}\}$ is periodic.
\begin{itemize}
\item[1.] For $t\geq 0$,
\begin{equation}
\label{pathsleq11}
w(\Pi_{ij,t}^{\mu})\leq\lambda_{\mu}^t (\nacht_{\mu}^{(t)})_{ij}
\end{equation}
\item[2.] For $t\geq T_{\mu}+2\tau_{\mu}(n_{\mu}-1)$ ,
\begin{equation}
\label{pathsgeq11}
w(\Pi_{ij,t}^{\mu})\geq\lambda_{\mu}^t(\nacht_{\mu}^{(t)})_{ij}
\end{equation}
\end{itemize}
\end{theorem}
\begin{proof}
We can w.l.o.g. assume that $\lambda_{\mu}=1$, since both
\eqref{pathsleq11} and \eqref{pathsgeq11} are stable under scalar
multiplication of $A$. After this, the claim follows from Theorem
\ref{pathology0}.
\end{proof}

In the theorem above, we can choose $T_{\mu}$ equal to each other
and of the order $O(n^2)$ for all $\mu$. The main result of this
section immediately follows now from Theorem~\ref{pathology1} and
Eqn.~\eqref{aijmut}, noting that $T_{\mu}+2\tau_{\mu}(n_{\mu}-1)\leq
3n^2$ for all $\mu$.

\begin{theorem}[Nachtigall expansion]
\label{t:nacht} Let $A\in\Rpnn$. Then for all $t\geq 3n^2$
\begin{equation}
\label{e:nachtrep} A_{\mu}^t=\bigoplus_{\nu=\mu}^m
\lambda_{\nu}^t\nacht_{\nu}^{(t)}.
\end{equation}
In particular,
\begin{equation}
\label{e:nachtrepa} A^t=\bigoplus_{\nu=1}^m
\lambda_{\nu}^t\nacht_{\nu}^{(t)}.
\end{equation}
\end{theorem}


\section{Ultimate expansion}
\label{s:ultexp} In this section we construct a different expansion
of $A^t$ which we call the ultimate expansion, in order to describe
the ultimate behavior of $A^t$. This expansion is related to the
Nachtigall expansion of Section~\ref{s:nachtexp} with the selection
rule $\crit_{\mu}:=\crit(A_{\mu})$. The latter expansion will be
called the {\em canonical Nachtigall expansion}.

\begin{figure}
\begin{tabular}{cccc}
\begin{tikzpicture}[scale=0.7]
\draw (-2.5,2.5) node[circle,draw]{$S_1$}; 
\begin{scope}[red]
\foreach \i in
{-1.5,-0.5, 0.5, 1.5,2.5} {
\draw (\i ,2.5) node{$*$};
}
\end{scope}

\begin{scope}[brown]
\foreach \i in {1.5,0.5}
{
\draw (-2.5,\i) node{$*$};
}

\foreach \i in {-0.5, -1.5,-2.5} { \draw (-2.5,\i) node{$0$};
}
\end{scope}

\foreach \i in {0.5,1.5,2.5}
{
\foreach \j in {1.5,0.5,-0.5,-1.5,-2.5}
{
\draw (\i,\j) node{$*$};
}
}

\foreach \i in {-1.5,-0.5}
{
\foreach \j in {1.5,0.5}
{
\draw (\i,\j) node{$*$};
}
}

\foreach \i in {-1.5,-0.5} 
{ 
\foreach \j in {-0.5,-1.5,-2.5} 
{ 
\draw(\i,\j) node{$0$}; } }

\draw (-3,3) -- (3,3);
\draw (-3,-3) -- (-3,3); \draw (-3,-3) -- (3,-3); \draw (3,-3) --
(3,3);

\begin{scope}[dashed,brown]
\draw (-2,3) -- (-2,-3);
\draw (-2.5,-3.3) node{$C_1$};
\end{scope}

\begin{scope}[dashed,red]
\draw (-3,2) -- (3,2);
\draw (3.4,2.5) node{$R_1$};
\end{scope}

\if{
\begin{scope}[dashed]
\foreach \j in {2.25,2.5,2.75}
{
\draw (-3,\j) -- (3,\j);
}
}\fi


\if{
\begin{scope}[dashed]
\foreach \i in {-2.25,-2.5,-2.75}
{
\draw (\i,-3) -- (\i,3);
}
}\fi


\if{
\begin{scope}[shorten >=1pt,->]
\draw (2,2.5) -- (4.5,2.5);
\end{scope}
\draw  (4.5,2.5) node[anchor=west]{$\lambda_1, C_1, S_1, R_1$};
}\fi

\end{tikzpicture}

&&

\begin{tikzpicture}[scale=0.7]
\draw (-2.5,2.5) node{$0$}; \foreach \i in {-1.5,-0.5, 0.5,
1.5,2.5} {
\draw (\i ,2.5) node{$0$}; }

\foreach \i in {1.5,0.5, -0.5, -1.5,-2.5} { \draw (-2.5,\i)
node{$0$};
}

\begin{scope}[red]
\foreach \i in {-0.5,0.5,1.5,2.5}
{
\draw (\i ,1.5) node{$*$};
}
\end{scope}

\begin{scope}[brown]
\foreach \i in {0.5} { \draw (-1.5,\i) node{$*$};
\foreach \i in {-0.5, -1.5,-2.5} 
{ \draw (-1.5,\i) node{$0$};}
}
\end{scope}

\foreach \i in {0.5} 
{ \draw (-0.5,\i) node{$*$};}

\foreach \i in {-0.5, -1.5,-2.5} { \draw (-0.5,\i) node{$0$};
}

\foreach \i in {0.5,1.5,2.5} { \foreach \j in {0.5,-0.5,-1.5,-2.5} {
\draw (\i,\j) node{$*$}; } } \draw (-1.5,1.5)
node[circle,draw]{$S_2$}; \draw (-3,3) -- (3,3); \draw (-3,-3)
-- (-3,3); \draw (-3,-3) -- (3,-3); \draw (3,-3) -- (3,3);

\begin{scope}[dashed,brown]
\draw (-1,2) -- (-1,-3);
\draw (-1.5,-3.3) node{$C_2$};
\end{scope}

\begin{scope}[dashed,red]
\draw (-2,1) -- (3,1);
\draw (3.4,1.5) node{$R_2$};
\end{scope}

\draw (-2,2) -- (-2,-3);
\draw (-2,2) -- (3,2);

\if{
\begin{scope}[shorten >=1pt,->]
\draw (2,1.5) -- (4.5,1.5);
\end{scope}
\draw  (4.5,1.5) node[anchor=west]{$\lambda_2, C_2, S_2, R_2$};
}\fi
\end{tikzpicture}\\

\begin{tikzpicture}[scale=0.7]
\draw (-2.5,2.5) node[circle,draw]{$S^{\ulabel}_1$};

\draw (-1.5,1.5) node{$\digr_1$};

\begin{scope}[red]
\foreach \i in {-1.5,-0.5, 0.5, 1.5,2.5}
{
\draw (\i ,2.5) node{$*$};
}
\end{scope}

\begin{scope}[brown]
\foreach \i in {1.5,0.5}
{
\draw (-2.5,\i) node{$*$};
}
\foreach \i in {-2.5} 
{ 
\foreach \j in {-0.5,-1.5,-2.5} 
{
\draw (\i,\j) node{$0$}; 
} 
}

\end{scope}

\foreach \i in {-0.5,0.5,1.5,2.5}
{
\draw (\i ,1.5) node{$*$};
}

\foreach \i in {0.5}
{
\draw (-1.5,\i) node{$*$};
}

\foreach \i in {-0.5,-1.5} { \foreach \j in {-0.5,-1.5,-2.5} {
\draw (\i,\j) node{$0$}; } }

\foreach \i in {0.5,1.5,2.5}
{
\foreach \j in {-0.5,-1.5,-2.5}
{
\draw (\i,\j) node{$*$};
}
}

\foreach \i in {-0.5,0.5,1.5,2.5}
{
\draw (\i, 0.5) node{$*$};
}

\draw (-3,3) -- (3,3);
\draw (-3,-3) -- (-3,3); \draw (-3,-3) -- (3,-3); \draw (3,-3) --
(3,3);
\begin{scope}[dashed]
\begin{scope}[brown]
\draw (-2,3) -- (-2,-3);
\draw (-2.5,-3.4) node {$C_1^{\circ}$};
\end{scope}
\begin{scope}[red]
\draw (-3,2) -- (3,2);
\draw (3.4,2.5) node {$R_1^{\circ}$};
\end{scope}
\draw (-1.5,1.5) circle (1.25cm);
\end{scope}
\if{
\begin{scope}[shorten >=1pt,->]
\draw (2,2.5) -- (4.5,2.5);
\end{scope}
\draw  (4.5,2.5) node[anchor=west]{$\lambda_1, C_1, S_1, R_1$};
}\fi
\end{tikzpicture}
&&
\begin{tikzpicture}[scale=0.7]
\foreach \i in {-2.5,-1.5,-0.5} { \foreach \j in
{2.5,1.5,0.5,-0.5,-1.5,-2.5} { \draw (\i,\j) node{$0$}; } }
\foreach \i in {0.5,1.5,2.5} { \foreach \j in {2.5,1.5, 0.5} { \draw
(\i,\j) node{$0$}; } } 

\begin{scope}[red]
\foreach \i in {1.5,2.5} {
\draw (\i ,-0.5) node{$*$};
}
\end{scope}

\begin{scope}[brown]
\foreach \i in {-1.5,-2.5}
{
\draw (0.5,\i) node{$*$};
}
\end{scope}

\draw (1.5,-1.5) node{$\digr_2$};
\draw (1.5,-2.5) node{$*$};
\draw (2.5,-1.5) node{$*$};
\draw (2.5,-2.5) node{$*$};

\draw (0.5,-0.5) node[circle,draw]{$S^{\ulabel}_2$}; \draw (-3,3) --
(3,3); \draw (-3,-3) -- (-3,3); \draw (-3,-3) -- (3,-3); \draw
(3,-3) -- (3,3);
\begin{scope}[dashed]
\begin{scope}[brown]
\draw (1,0) -- (1,-3);
\draw (0.5,-3.4) node {$C_2^{\circ}$};
\end{scope}
\begin{scope}[red]
\draw (0,-1) -- (3,-1);
\draw (3.4,-0.5) node {$R_2^{\circ}$};
\end{scope}
\draw (1.5,-1.5) circle (1.25cm);
\end{scope}

\draw (0,0) -- (0,-3);
\draw (0,0) -- (3,0);
\if{
\begin{scope}[shorten >=1pt,->]
\draw (2,-0.5) -- (4.5,-0.5);
\end{scope}
\draw  (4.5,-0.5) node[anchor=west]{$\lambda^{\circ}_2, C^{\circ}_2, S^{\circ}_2, R^{\circ}_2$};
}\fi
\end{tikzpicture}
\end{tabular}
\caption{Formation of first two terms in a Nachtigall expansion (upper part)
 and ultimate expansion (lower part): a schematic example. The associated
digraph consists of two components, denoted $\digr_1$ and $\digr_2$.}
\label{f:nachtult}
\end{figure}
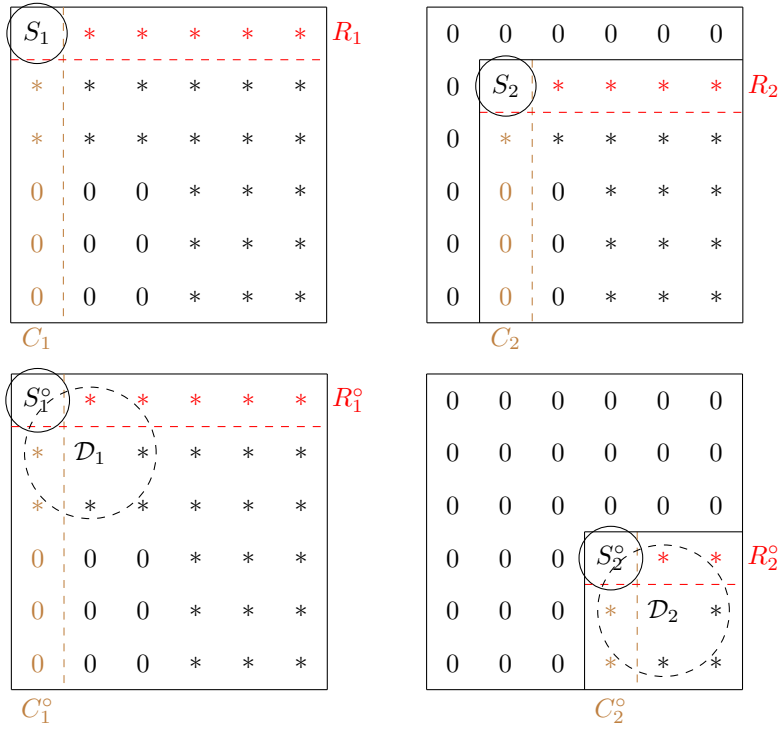

The elements of the ultimate expansion will be labeled by $\ulabel$,
since we need to distinguish them from those of the canonical
Nachtigall expansion. For instance, we will write
$A_{\mu}^{\ulabel}$ versus $A_{\mu}$, $\crit_{\mu}^{\ulabel}$ versus
$\crit_{\mu}$ and $\lambda_{\mu}^{\ulabel}$ versus $\lambda_{\mu}$,
etc.

Let $A=(a_{ij})\in\Rpnn$. Define $\lambda^{\ulabel}_1=\lambda(A)$,
let $\crit^{\ulabel}_1=(N_1^{\ulabel},E_1^{\ulabel})$ be the
critical graph of $A$ and denote by $M_1^{\ulabel}$ the set of nodes
in all components of $\digr(A)$ that contain the components of
$\crit^{\ulabel}_1$. Set $A^{\ulabel}_1:=A$ and $K_1^{\ulabel}:=N$.

By induction for $\mu\geq 2$, define $K^{\ulabel}_{\mu}:=N\bez
\cup_{i=1}^{\mu-1} M_i^{\ulabel}$ (instead of $K_{\mu}=N\bez
\cup_{i=1}^{\mu-1} N_i$ in the case of a Nachtigall expansion), and
define $A^{\ulabel}_{\mu}=(a^{\mu\ulabel}_{ij})\in\Rpnn$ by

\begin{equation}
\label{amu1} a^{\mu\ulabel}_{ij}=
\begin{cases}
a_{ij}, & \text{if $i,j\in K^{\ulabel}_{\mu}$,}\\
0, & \text{otherwise.}
\end{cases}
\end{equation}

Define $\lambda^{\ulabel}_{\mu}:=\lambda(A^{\ulabel}_{\mu})$. If
$\lambda^{\ulabel}=0$ then stop, otherwise let
$\crit^{\ulabel}_{\mu}=(N_{\mu}^{\ulabel},E_{\mu}^{\ulabel})$ be the
critical graph of $A_{\mu}$, let $M_{\mu}^{\ulabel}$ be the set of
all nodes in the components of $\digr(A)$ which contain the
components of $\crit_{\mu}^{\ulabel}$, and proceed with the above
definition for $\mu:=\mu+1$.

By the above procedure we define $K_{\mu}^{\ulabel}$,
$A_{\mu}^{\ulabel}$, $\lambda_{\mu}^{\ulabel}$,
$\crit_{\mu}^{\ulabel}=(N_{\mu}^{\ulabel}, E_{\mu}^{\ulabel})$ and
$M_{\mu}^{\ulabel}$ for $\mu=1,\ldots,m^{\ulabel}$, where
$m^{\ulabel}\leq n$ is the last number $\mu$ such that
$\lambda_{\mu}^{\ulabel}\neq 0$. Note that
$\{\lambda_{\mu}^{\ulabel},\;\mu=1,\ldots,m^{\ulabel}\}$ is the set
of m.c.g.m. of all nontrivial components of $\digr(A)$, and each
$\crit_{\mu}^{\ulabel}$ consists of the critical digraphs of
(possibly several) such components.

Let $\gamma_{\mu}^{\ulabel}$ be the cyclicity of
$\crit_{\mu}^{\ulabel}$, and let $B_{\mu}^{\ulabel}$,
$C_{\mu}^{\ulabel}$, $S_{\mu}^{\ulabel}$ and $R_{\mu}^{\ulabel}$ be
defined from
$(A_{\mu}^{\ulabel}/\lambda_{\mu}^{\ulabel})^{\gamma_{\mu}^{\ulabel}}$
in full analogy with $B_{\mu}$, $C_{\mu}$, $S_{\mu}$ and $R_{\mu}$
in Sect. \ref{s:nachtexp}. The matrices $S_{\mu}^{\ulabel}$ can be
again simultaneously scaled to $0-1$ form.

Essentially in the new construction we contract by the components of
$\digr(A)$ instead of the components of $\crit(A_{\mu})$ in the case
of the canonical Nachtigall expansion. See Figure~\ref{f:nachtult} for a visual 
comparison.

Denote by  $\lambda(i)$ the m.c.g.m. of the component of $\digr(A)$
to which $i$ belongs, and let $\lambda(i)=0$ if $\{i\}$ is a trivial component
of $\digr(A)$. For a path $P$ define $\lambda(P):=
\max_{i\in N_P} \lambda(i)$ where $N_P$ is the set of nodes visited
by $P$. Recall that $\Pi_{ij,t}$ denotes the set of paths on
$\digr(A)$ which connect $i$ to $j$ and have length $t$. Denote by
$\Pi_{ij,t}^{\mu\ulabel}$ the set of paths $P\in\Pi_{ij,t}$ such
that $P$ contains a node in $N_{\mu}^{\ulabel}$ and
$\lambda(P)=\lambda_{\mu}$. Such paths will be called {\em
$\mu$-hard}. Note that they are $\crit(A_{\mu}^{\ulabel})$-heavy
with respect to $A_{\mu}^{\ulabel}$.

Denote
\begin{equation}
\ultim_{\mu}^{(t)}:=C_{\mu}^{\ulabel}(S_{\mu}^{\ulabel})^tR_{\mu}^{\ulabel}.
\end{equation}

These $CSR$ products are defined from $A_{\mu}^{\ulabel}$ and
$\crit_{\mu}^{\ulabel}$ in the same way as $\proj^{(t)}$ were
defined in Sect.~\ref{s:projector} from $A$ and $\crit$, and it
follows that the sequence $\{\ultim_{\mu}^{(t)},\;t\geq 0\}$ has the
period $\gamma_{\mu}^{\ulabel}$. Denote by $\tau_{\mu}^{\ulabel}$
the greatest cyclicity of a component in $\crit_{\mu}^{\ulabel}$ and
by $n_{\mu}^{\ulabel}$ the number of nodes in $K_{\mu}^{\ulabel}$.
The next result follows from Theorem \ref{pathology0}.

\begin{theorem}[$\mu$-Hard paths]
\label{pathology2} Let $T_{\mu}^{\ulabel}$ be such that
$\{(S_{\mu}^{\ulabel})^t,\;t\geq T_{\mu}^{\ulabel}\}$ is periodic.
\begin{itemize}
\item[1.] For $t\geq 0$,
\begin{equation}
\label{pathsleq22}
w(\Pi_{ij,t}^{\mu\ulabel})\leq(\lambda_{\mu}^{\ulabel})^t
(\ultim_{\mu}^t)_{ij}.
\end{equation}
\item[2.] For $t\geq T_{\mu}^{\ulabel}+2\tau_{\mu}^{\ulabel}(n_{\mu}^{\ulabel}-1)$ ,
\begin{equation}
\label{pathsgeq22}
w(\Pi_{ij,t}^{\mu\ulabel})\geq(\lambda_{\mu}^{\ulabel})^t
(\ultim_{\mu}^t)_{ij}.
\end{equation}
\end{itemize}
\end{theorem}



Comparing the constructions above with those of the canonical
Nachtigall expansion (see Sect. \ref{s:nachtexp} assuming that
$\crit_{\mu}:=\crit(A_{\mu})$), we see that $\crit_1$ is the same as
$\crit_{1}^{\ulabel}$ and $\lambda_1$ is the same as
$\lambda_{1}^{\ulabel}$, however, other components and values may be
not the same. We next describe relation between them.

\begin{proposition}
\label{compsrel} Each $\nu=1,\ldots,m^{\ulabel}$ corresponds to a
unique $\mu=1,\ldots,m$ such that
$\lambda_{\nu}^{\ulabel}=\lambda_{\mu}$ and all components of
$\crit_{\nu}^{\ulabel}$ are also components of $\crit_{\mu}$.
\end{proposition}
\begin{proof}
Consider the canonical Nachtigall expansion. Note that
$\lambda_{\mu}$ strictly decrease, as at each step of the definition
we remove the whole critical digraph. Now pick arbitrary
$\lambda_{\nu}^{\ulabel}$, which is the m.c.g.m. of some component
of $\digr(A)$. There is a reduction step when $\crit_{\mu}$ for the
first time intersects with a component of $\digr(A)$ whose m.c.g.m.
is $\lambda_{\nu}^{\ulabel}$. Then
$\lambda_{\mu}=\lambda_{\nu}^{\ulabel}$, and $\crit_{\mu}$ has to
contain all components of $\crit_{\nu}^{\ulabel}$, precisely as they
are. This proves the claim.
\end{proof}

Further we renumber $\lambda_{\nu}^{\ulabel}$  so that $\nu=\mu$ if
$\lambda_{\nu}^{\ulabel}=\lambda_{\mu}$, meaning that the numbering
of $\lambda_{\nu}^{\ulabel}$ is adjusted to that of $\lambda_{\mu}$.
This defines a subset $\Sigma$ of $\{1,\ldots,m\}$, such that
$\lambda_{\mu}=\lambda_{\mu}^{\ulabel}$ if and only if
$\mu\in\Sigma$.

\begin{corollary}
\label{c:gammas}  $\gamma_{\mu}$ is a multiple of
$\gamma_{\mu}^{\ulabel}$ for each $\mu\in\Sigma$.
\end{corollary}
\begin{proof}
With the new numbering, all components of $\crit_{\mu}^{\ulabel}$
are also components of $\crit_{\mu}$ by Proposition~\ref{compsrel}.
\end{proof}

Unlike $\mu$-heavy paths, $\mu$-hard paths do not cover the whole
path sets $\Pi_{ij,t}$ in general. However evidently
\begin{equation}
\label{e:heavymeaning}
(A_{\mu}^{\ulabel})^t_{ij}=w(\Pi_{ij,t}^{\mu\ulabel}),\quad\text{if
$i\in N_{\mu}^{\ulabel}$ or $j\in N_{\mu}^{\ulabel}$.}
\end{equation}

From this and Theorem~\ref{pathology2} we deduce the following.

\begin{proposition}
\label{t:nacht2} Let $A\in\Rpnn$. For all $t\geq 3n^2$
\begin{equation}
\label{e:nacht22}
(A_{\mu}^{\ulabel})^t_{ij}=\lambda_{\mu}^t\ultim_{\mu}^{(t)},
\quad\text{if $i\in N_{\mu}^{\ulabel}$ or $j\in N_{\mu}^{\ulabel}$.}
\end{equation}
\end{proposition}


For the sequel we need to establish some relation between
connectivity on $\digr(A)$ and nonzero entries of
$\ultim_{\mu}^{(t)}$.

We denote by $\gamma^{\ulabel}$ the l.c.m. of all cyclicities
$\gamma_{\mu}^{\ulabel}$ of all components $\crit_{\mu}^{\ulabel}$.
Recall that we denote by $N_P$ the set of nodes visited by a path
$P$ and by $\lambda(P)$ the greatest m.c.g.m. of a component visited
by $P$.

\begin{proposition}
\label{paths-entries} Let $i,j\in N,$ $l\geq 0$ and $\mu\in\Sigma$.
The following are equivalent.
\begin{itemize}
\item[1.]  $(\ultim_{\mu}^{(l)})_{ij}\neq 0$;
\item[2.] For all $t\equiv l(\modd\;\gamma^{\ulabel})$ such that
$t\geq 3n^2$, there is a $\mu$-hard path of length $t$ connecting
$i$ to $j$;
\item[3.] For some $t\equiv l(\modd\;\gamma^{\ulabel})$ there exists a
$\mu$-hard path of length $t$ connecting $i$ to $j$;
\item[4.] For some $t\equiv l(\modd\;\gamma^{\ulabel})$ there exists a
path $P$ of length $t$ such that $\lambda(P)=\lambda_{\mu}$.
\end{itemize}
\end{proposition}
\begin{proof}
Implications 1.$\Leftrightarrow$2. and 3$\Leftrightarrow$1. follow
from Theorem \ref{pathology2} and the periodicity of
$\ultim_{\mu}^{(t)}$. Implications 2.$\Rightarrow$3.$\Rightarrow$4.
are evident. It remains to prove 4.$\Rightarrow$3. Let $k$ be a node
in $N_P$ which belongs to $M_{\mu}^{\ulabel}$, and let $l$ be a node
in $N_{\mu}^{\ulabel}$ (that is, a critical node) in the same
component of $\digr(A)$ as $k$. There exists a cycle containing both
$k$ and $l$. Adjoining $\gamma^{\ulabel}$ copies of this cycle to
$P$ we obtain a $\mu$-hard path, whose length is congruent to
$l(\modd\;\gamma^{\ulabel})$.
\end{proof}

Now we establish the ultimate expansion of matrix powers, as an
ultimate form of the canonical Nachtigall expansion. We will write
$a(t)\ulteq b(t)$ if $a(t)=b(t)$ for all $t\geq t'$ where $t'$ is an
unknown integer, and analogously for inequalities.

\begin{theorem}[Ultimate expansion]
\label{t:serg-sch} Let $A\in\Rpnn$. For all $\mu\in\Sigma$
\begin{equation}
\label{e:ultim} (A_{\mu}^{\ulabel})^t\ulteq
\bigoplus_{\nu\in\Sigma\colon\nu\geq\mu} \lambda^t_{\nu}
\ultim_{\nu}^{(t)}.
\end{equation}
In particular,
\begin{equation}
\label{e:ultima} A^t\ulteq \bigoplus_{\nu\in\Sigma} \lambda^t_{\nu}
\ultim_{\nu}^{(t)}.
\end{equation}
\end{theorem}
\begin{proof}
It suffices to prove \eqref{e:ultima}. First note that
$\ultim_{\mu}^{(t)}\leq\nacht_{\mu}^{(t)}$ for all $\mu\in\Sigma$,
since any $\mu$-hard path is a $\mu$-heavy path. As $A^t\ulteq
\bigoplus_{\mu}\lambda_{\mu}^t\nacht_{\mu}^{(t)}$ by the canonical
Nachtigall expansion, it suffices to prove that
\begin{equation}
\label{nacultineq}
\bigoplus_{\mu=1}^m\lambda_{\mu}^t\nacht_{\mu}^{(t)}\ultleq\bigoplus_{\mu\in\Sigma}
\lambda_{\mu}^t\ultim_{\mu}^{(t)}
\end{equation}
For all $\mu$, $i$ and $j$ such that $(\nacht_{\mu}^{(t)})_{ij}\neq
0$, we will show that either $\mu\in\Sigma$ and
$(\ultim_{\mu}^{(t)})_{ij}=(\nacht_{\mu}^{(t)})_{ij}$, or there
exists $\nu\in\Sigma$ such that $\lambda_{\nu}>\lambda_{\mu}$ and
$(\ultim_{\nu}^{(t)})_{ij}\neq 0$.

Assume that either $\mu\notin\Sigma$, or $\mu\in\Sigma$ but
$(\nacht_{\mu}^{(t)})_{ij}>(\ultim_{\mu}^{(t)})_{ij}$. Theorem
\ref{pathology1} implies that for all $l$ such that $l\geq 3n^2$ and
$l\equiv t(\modd\;\gamma_{\mu})$ there exist paths
$P\in\Pi_{ij,l}^{\mu}$ such that
$w(P)=\lambda_{\mu}^l(\nacht_{\mu}^{(t)})_{ij}$. We are going to show 
that these paths are not $\mu$-hard. If $\mu\notin\Sigma$ then this is immediate.
If $\mu\in\Sigma$
then by Corollary~\ref{c:gammas} $\gamma_{\mu}$ is a multiple of
$\gamma_{\mu}^{\ulabel}$, and hence $l\equiv
t(\modd\;\gamma_{\mu}^{\ulabel})$.  If $P$ is $\mu$-hard,
then 
$w(P)\leq\lambda_{\mu}^l
(\ultim_{\mu}^{(t)})_{ij}$ by Theorem \ref{pathology2}, which
implies $(\ultim_{\mu}^{(t)})_{ij}\geq(\nacht_{\mu}^{(t)})_{ij}$
contradicting our assumptions. Hence $P$ are not $\mu$-hard, meaning that for any 
such path there exists
$\nu\in\Sigma$ such that $\lambda_{\nu}=\lambda(P)>\lambda_{\mu}$.
Applying Proposition \ref{paths-entries}, we obtain that
$(\ultim_{\nu}^{(t)})_{ij}\neq 0$ with
$\lambda_{\nu}>\lambda_{\mu}$. The claim is proved.
\end{proof}

If $A$ is irreducible, then the ultimate expansion has only one
term, which corresponds to its critical graph $\crit(A)$. In
general, it has several terms (up to $n$) corresponding to the
critical graphs of the components of $\digr(A)$ (or possibly clusters
of critical graphs of the components with the same m.c.g.m.) Thus, Theorem
\ref{t:serg-sch} can be regarded as a generalization of the
Cyclicity Theorem, see \cite{HOW:05} Theorem 3.9 or \cite{BCOQ}
Theorem 3.108, which it implies as a special irreducible case.

\section{Computational complexity}
\label{s:comp}

Given $A\in\Rpnn$,
we investigate the computational complexity of the following problems.\\
(P1) For given $t$, reconstruct all terms $\nacht_{\mu}^{(t)}$ of a
Nachtigall expansion with a prescribed selection rule for
$\crit_{\mu}$.\\
(P2) For given $t$: $0\leq t<\gamma$, reconstruct all terms
$\ultim_{\mu}^{(t)}$ of the ultimate expansion.

In (P1) we assume that selecting the subdigraph $\crit_{\mu}$ of
$\crit(A_{\mu})$ does not take more than $O(n^3)$ operations. This
holds in particular if $\crit_{\mu}$ is an arbitrary cycle of
$\crit(A_{\mu})$ as in \cite{Mol-03,Nacht}.

Problem (P1) is close to the problem considered by Moln\'{a}rov\'{a}
\cite{Mol-03}, and Problem (P2) is extension of a problem regarded
by Sergeev \cite{Ser-09}. An $O(n^4\log n)$ solution of these
problems is given below. It is based on visualisation, square
multiplication and permutation of cyclic classes. See
Seman\v{c}\'{\i}kov\'a~\cite{Sem-06,Sem-07} for closely related
studies in max-min algebra.

\begin{theorem}
\label{t:compcomp} For any $A\in\Rpnn$, problems (P1) and (P2) can
be solved in $O(n^4\log n)$ operations.
\end{theorem}
\begin{proof}
(P1): First we need to compute $\lambda_{\mu}$, $A_{\mu}$ and
$\crit_{\mu}$ for all $\mu$. At each step the computation requires
no more than $O(n^3)$ operations, based on Karp and Floyd-Warshall
methods applied to each component of $\digr(A_{\mu})$. The total
complexity is no more than $O(n^4)$. After this, we find all cyclic
classes in each $\crit_{\mu}$, which has total complexity $O(n^2)$,
and hence the cyclicities $\gamma_{\sigma}$ of all components of the
graphs $\crit_{\mu}$. At this stage we can also find a scaling which
leads to a total $S$-visualization of $A$ (and hence all $A_{\mu}$).
This relies on Floyd-Warshall method applied to each $S_{\mu}$ and
takes no more than $O(n^3)$ operations in total.

By Theorem~\ref{t:nacht}, $A_{\mu}^t$ admit Nachtigall expansion for
all $\mu$ and all $t\geq 3n^2$. The rows and columns of $A_{\mu}^t$
with indices in $N_{\mu}$ are determined at $t\geq 3n^2$ only by
$\nacht_{\mu}^t=C_{\mu}S_{\mu}^tR_{\mu}$, since by construction
these rows and columns are zero in all terms $\nacht_{\nu}^{(t)}$
for $\nu>\mu$. This means in particular that these rows and columns
become periodic after $3n^2$ time. By repeated squaring $A_{\mu},
A_{\mu}^2, A_{\mu}^4,\ldots,$ we reach a power $A_{\mu}^r$ with
$r\geq 3n^2$, which requires no more than $O(n^3\log n)$ operations.
Now we can use Corollary \ref{c:per} identifying $C_{\mu}S_{\mu}^r$
and $S_{\mu}^rR_{\mu}$ as submatrices extracted from columns, resp.
rows, of $A_{\mu}^t$ with indices in $N_{\mu}$. By Theorem
\ref{t:rotate} we can obtain $S_{\mu}^tR_{\mu}$ from
$S_{\mu}^rR_{\mu}$ and $C_{\mu}$ from $C_{\mu}S_{\mu}^r$ by the
permutation on cyclic classes determined by the remainders
$r(\modd\;\gamma_{\sigma})$ and $(t-r)(\modd\;\gamma_{\sigma})$, for
each cyclicity $\gamma_{\sigma}$ of a component of $\crit_{\mu}$.
This takes $O(n^2)$ overrides. Finally we compute
$\nacht_{\mu}^t=C_{\mu}S_{\mu}^t R_{\mu}$ ($O(n^3)$ operations). We
conclude that the total complexity for all $\mu$ does not exceed
$O(n^4\log n)$ operations.

(P2): It is clear that the computation of all prerequisites for the
ultimate expansion is done like in the first para of the proof of
(P1), and takes no more than $O(n^4)$ operations. After that we use
Proposition~\ref{t:nacht2} which means that the critical rows and
columns in each $(A_{\mu}^{\ulabel})^r$ for $r\geq 3n^2$ are
determined only by $\ultim_{\mu}^{(r)}$. Hence the factors of each
$\ultim_{\mu}^{(t)}$ can be computed by matrix squaring of
$A_{\mu}^{\ulabel}$, followed by a permutation on cyclic classes and
matrix multiplication, which overall takes no more than $O(n^3\log
n)$ operations. We conclude that the total complexity for all $\mu$
does not exceed $O(n^4\log n)$ operations.
\end{proof}

\section{Orbit periodic matrices}
\label{s:totalper} Being motivated by the results of
Butkovi\v{c}~et~al.~\cite{BCG} on robust matrices, we are going to
derive necessary and sufficient conditions for orbit periodicity and
to show that they can be verified in polynomial time.

A matrix $A\in\Rpnn$ is called {\em orbit periodic} if for each
$y\in\Rpn$ there exists $\lambda(y)\in\Rpn$ such that
$A^{t+\gamma^{\ulabel}}y=(\lambda(y))^{\gamma^{\ulabel}}A^ty$ for
all sufficiently large $t$, where (as above)$\gamma^{\ulabel}$ is
the joint cyclicity (l.c.m.) of the critical graphs of all
components of $\digr(A)$.

A sequence $\{A^ty,\;t\geq 0\}$ with the above property will be
called {\em ultimately linear periodic} and $\lambda(y)$ will be
called its {\em ultimate growth rate}. The same wording will be used
for the sequences $\{a_{ij}^t,\;t\geq 0$. We say that a subsequence
$\{a_{ij}^{l+\gamma^{\ulabel}s},\; s\geq 0\}$ has ultimate growth
rate $\lambda$, if there exists $\alpha_{ij}\neq 0$ such that
$a_{ij}^{l+\gamma^{\ulabel}s}=\alpha_{ij}\lambda^{l+\gamma^{\ulabel}s}$
for all $s$ starting from a sufficiently large number.

It may seem more general if in the above definition of linear
periodicity we replace $\gamma^{\ulabel}$ by $\gamma(y)$. But using
the ultimate expansion \eqref{e:ultima} we conclude that
$\{A^ty,\;t\geq 0\}$ is ultimately linear periodic if and only if
there exists $\mu\in\Sigma$ such that
$A^ty\ulteq\lambda_{\mu}^t\ultim_{\mu}^{(t)}y$. As
$\ultim_{\mu}^{(t+\gamma^{\ulabel})}=\ultim_{\mu}^{(t)}$ for all
$\mu$ and $t$, we conclude that the exact period of $\{A^ty,\;t\geq
0\}$ has to divide $\gamma^{\ulabel}$.

The ultimate expansion leads to the following properties of the
sequences $\{a_{ij}^t,\; t\geq 0\}$, already known in max algebra
\cite{BdS,Gav:04,Mol-05}. As above, $\lambda(P)$ is the largest
m.c.g.m. of the components of $\digr(A)$ visited by $P$, and
$\Pi_{ij,t}$ denotes the set of paths of length $t$ connecting $i$
to $j$.

\begin{lemma}
\label{l:genper} For each $l\colon 0\leq l<\gamma^{\ulabel}$, the
subsequence $\{a_{ij}^{l+\gamma^{\ulabel}s},\; s\geq 0\}$ is
ultimately zero or has an ultimate growth rate $\lambda_{ij}(l)$. In
the latter case, for each such $l$ and each sufficiently large
$t\equiv l(\modd\;\gamma^{\ulabel})$ there exists $P\in\Pi_{ij,t}$
such that $\lambda(P)=\lambda_{ij}(l)$.
\end{lemma}
\begin{proof}
The ultimate expansion~\eqref{e:ultima} implies that for each $l$
there exists $\mu\in\Sigma$ such that $a_{ij}^{l+\gamma^{\ulabel}s}=
\lambda_{\mu}^{l+\gamma^{\ulabel}s}(\ultim_{\mu}^{(l)})_{ij}$ at
sufficiently large $s$, so the subsequence has growth rate
$\lambda_{ij}(l):=\lambda_{\mu}$. The second part of the statement
follows from Proposition \ref{paths-entries}.
\end{proof}

\begin{lemma}
\label{l:paths}
 If $P\in\Pi_{ij,l}$ then
 $\{a_{ij}^{l+\gamma^{\ulabel}s},\; s\geq 0\}$
 has ultimate growth rate at least $\lambda(P)$.
\end{lemma}
\begin{proof}
Using Proposition \ref{paths-entries} cond. 4, we obtain that
($\ultim_{\mu}^{(l)})_{ij}\neq 0$ for $\mu$ such that
$\lambda_{\mu}=\lambda(P)$, and then
$a_{ij}^{l+\gamma^{\ulabel}s}\geq
(\lambda(P))^{l+\gamma^{\ulabel}s}(\ultim_{\mu}^{(l)})_{ij}$ at
sufficiently large $s$.
\end{proof}

Denote by $L^{\ulabel}$ the set of nodes in the nontrivial
components of $\digr(A)$. The next statement follows from the
Cyclicity Theorem \cite{BCOQ,HOW:05}. For the sake of completeness
we deduce it from the ultimate expansion.

\begin{lemma}
\label{l:samecomp}
 If $i,j\in L^{\ulabel}$ belong to the same component of $\digr(A)$, then
 $\{a_{ij}^t,\;t\geq 0\}$ is
 ultimately linear periodic and its growth rate is the m.c.g.m. of that
 component. Further, for each $i\in L^{\ulabel}$ and all $t$
 there exist $k$ and $l$ in the same
 component of $\digr(A)$ such that $a_{ik}^t\neq 0$ and
 $a_{li}^t\neq 0$.
\end{lemma}
\begin{proof}  All paths which connect $i$ to $j$ are
$\mu$-hard, with $\lambda_{\mu}=\lambda(i)=\lambda(j)$. Using
Theorem~\ref{pathology2} we obtain that only
$(\ultim_{\mu}^{(t)})_{ij}$ is nonzero and hence the ultimate
expansion reads
$a^t_{ij}\ulteq\lambda_{\mu}^t(\ultim_{\mu}^{(t)})_{ij}$. For the
second part note that $i$ belongs to a cycle with nonzero weight.
\end{proof}

If $i$ is connected to $j$ by a path, we denote this by $i\connect
j$. Observe that if $i\connect j$ then also $k\connect l$ for each
$k$ in the same component of $\digr(A)$ with $i$ and for each $l$ in
the same component of $\digr(A)$ with $j$. We also denote
$i\leftrightarrow j$ if both $i\connect j$ and $j\connect i$ (i.e.,
if $i$ and $j$ are in the same component of $\digr(A)$). In the next
theorem we describe, in terms of such relations, when the sequences
of columns $\{A^te_i,\;t\geq 0\}$ are ultimately linear periodic.

\begin{proposition}
\label{per-cols} Let $A\in\Rpnn$ and $j\in L^{\ulabel}$. The
sequence $\{A^te_j,\; t\geq 0\}$ is ultimately linear periodic if
and only if for all $i\in L^{\ulabel}$, $i\connect j$ implies
$\lambda(i)\leq\lambda(j)$.
\end{proposition}

\begin{proof}
The ``only if'' part: Lemma~\ref{l:samecomp} implies that for each
$t$ there exists $k$ such that $a_{kj}^t\neq 0$, and the sequence
$\{a_{kj}^{t+s\gamma^{\ulabel}},\; s\geq 0\}$ has ultimate growth
rate $\lambda(j)$. If the condition does not hold, there exists a
path $P$ leading from $i$ to $j$ such that
$\lambda(P)=\lambda(i)>\lambda(j)$, and by Lemma~\ref{l:paths} there
is a subsequence of $\{a_{ij}^l,\; l\geq 0\}$ with ultimate growth
rate at least $\lambda(i)$.

The ``if'' part: If the sequence $\{A^te_j,\; t\geq 0\}$ is not
ultimately linear periodic, then some of its entries by Lemma~\ref{l:samecomp} have ultimate
growth rate $\lambda(j)$ and there is a subsequence of
$\{a_{kj}^t,\;t\geq 0\}$, for some $k\in N$, which has a different
ultimate growth rate. Lemma~\ref{l:genper} implies that this growth
rate has to be greater than $\lambda(j)$, and must be the m.c.g.m.
of a component which has access to $j$.
\end{proof}

We denote $i\strconnect j$ and say that $i$ {\em strongly accesses}
$j$, if $i$ can be connected to $j$ by a path of any length starting
from a certain number $T_{ij}$. For example, $i$ strongly accesses
$j$ if $i\connect j$ and the cyclicities of the components of
$\digr(A)$ containing $i$ and $j$ are coprime.  Observe that if
$i\strconnect j$ then also $k\strconnect l$ for each
$k\leftrightarrow i$ and $l\leftrightarrow j$. Now we show that
strong access relations are essential for the ultimate linear
periodicity of all sequences $\{A^t(e_i\oplus e_j),\;t\geq 0\}$.

\begin{proposition}
\label{per-2cols} Let $A\in\Rpnn$ and $i,j\in L^{\ulabel}$. Suppose
that $\{A^te_i,\; t\geq 0\}$ and $\{A^te_j,\; t\geq 0\}$ are
ultimately linear periodic. Then $\{A^t(e_k\oplus e_l),\;t\geq 0\}$
is also ultimately linear periodic for all $k\leftrightarrow i$ and
$l\leftrightarrow j$, if and only if either $i\strconnect j$, or
$j\strconnect i$, or both are false but $\lambda(i)=\lambda(j)$.
\end{proposition}
\begin{proof}
If $\lambda(i)=\lambda(j)$ then both $\{A^te_k,\; t\geq 0\}$ and
$\{A^te_l,\; t\geq 0\}$ for all $k\leftrightarrow i$ and
$l\leftrightarrow j$ have this growth rate and $\{A^t(e_k\oplus
e_l),\;t\geq 0\}$ is ultimately linear periodic with this growth
rate. So it remains to consider the case $\lambda(i)<\lambda(j)$. In
this case Proposition \ref{per-cols} implies $j\not\connect i$, and
therefore we have to show that $\{A^t(e_k\oplus e_l),\;t\geq 0\}$
are ultimately linear periodic for all $k\leftrightarrow i$ and
$l\leftrightarrow j$ if and only if $i\strconnect j$.

The ``if'' part:  For each $t\colon 0\leq t<\gamma^{\ulabel}$, if
there exists $m\in N$ and $s_1\geq 0$ such that
$a_{mk}^{t+s_1\gamma^{\ulabel}}\neq 0$, then there exists a path
$P\in\Pi_{mk,t+s_1\gamma^{\ulabel}}$. As $k\strconnect l$, this path
can be joined with a path from $k$ to $l$ of length
$s_2\gamma^{\ulabel}$, and we get that
$a_{ml}^{t+(s_1+s_2)\gamma^{\ulabel}}\neq 0$. Using
Lemma~\ref{l:paths} we obtain $a_{ml}^{t+s\gamma^{\ulabel}}\neq 0$
for all sufficiently large $s$ and it dominates over
$a_{mk}^{t+s\gamma^{\ulabel}}$ since it has larger growth rate.

The ``only if'' part: The ultimate linear periodicity of
$\{A^t(e_k\oplus e_j),\;t\geq 0\}$, for any $k\leftrightarrow i$,
implies that $\supp(A^t e_k)\subseteq\supp(A^t e_j)$ for all large
enough $t$. By Lemma~\ref{l:samecomp} there is $k\leftrightarrow i$
such that $(A^t)_{ik}\neq 0$, hence also $(A^t)_{ij}\neq 0$. As we
reasoned for any $t$, it follows that $i\strconnect j$.
\end{proof}

\begin{theorem}[Orbit periodicity]
\label{t:totalper} $A\in\Rpnn$ is orbit periodic if and only if the
following conditions hold for all $i,j\in L^{\ulabel}$:
\begin{itemize}
\item[1.] $i\connect j$ implies $\lambda(i)\leq\lambda(j)$,
\item[2.] if neither $j\strconnect i$ nor $i\strconnect j$ then
$\lambda(i)=\lambda(j)$,
\end{itemize}
or equivalently if $\{A^t(e_i\oplus e_j),\; t\geq 0\}$ are
ultimately linear periodic for all $i,j\in L^{\ulabel}$.
\end{theorem}
\begin{proof} We need only prove that 1. and 2. are sufficient for
orbit periodicity, the rest relies on Propositions \ref{per-cols}
and \ref{per-2cols}.

Let $\digr(A,y)$ be the subgraph induced by the set of nodes that
have access to $\supp(y):=\{i\colon y_i\neq 0\}$. Nontrivial
components $\digr_{\sigma}$ of $\digr(A,y)$ are ordered by relation
$\digr_{\sigma_1}\preceq\digr_{\sigma_2}$ if $i\strconnect j$ for
some (and hence all) $i\in\digr_{\sigma_1}$ and
$j\in\digr_{\sigma_2}$. Consider the maximal components with respect
to this relation, by conditions 1. and 2. they must have the same
m.c.g.m. and it must be the greatest one. We denote this m.c.g.m. by
$\lambda$ and show that it is the ultimate growth rate of $A^ty$.

By Lemma~\ref{l:genper} for each $l\colon 0\leq l<\gamma$, the
subsequence $\{a_{ij}^{l+s\gamma^{\ulabel}},\; s\geq 0\}$ has a
certain ultimate growth rate if it is not ultimately zero. We have
to show that $\lambda$ is the maximal growth rate of
$\{a_{ik}^{l+s\gamma^{\ulabel}},\; s\geq 0\}$ over $k\in\supp(y)$,
for every fixed $i\in N$ and $l\colon 0\leq l<\gamma^{\ulabel}$.
Then it follows that
$A^{t+\gamma^{\ulabel}}y\ulteq\lambda^{\gamma^{\ulabel}}A^ty$.

To avoid trivialities we assume that there exists $k\in\supp(y)$ and
such that $\{a_{ik}^{l+s\gamma^{\ulabel}},\;s\geq 0\}$ is not
ultimately zero. Then for some $k\in\supp(y)$ there exists a path
$P\in\Pi_{ik,t}$ where $t\equiv l(\modd\; \gamma^{\ulabel})$ which
visits a nontrivial component $\digr_{\sigma_1}$ of $\digr(A,y)$. If
the m.c.g.m. of $\digr_{\sigma_1}$ is $\lambda$, then by
Lemma~\ref{l:paths} the growth rate of
$\{a_{ik}^{l+s\gamma^{\ulabel}},\; s\geq 0\}$ is not less than
$\lambda$, hence it must be $\lambda$ and we are done. If the
m.c.g.m. of $\digr_{\sigma_1}$ is less than $\lambda$, then
$\digr_{\sigma_1}$ strongly accesses some component
$\digr_{\sigma_2}$ with m.c.g.m. $\lambda$, and $\digr_{\sigma_2}$
accesses a node $k'\in\supp(y)$. Due to the strong access we can
adjust the length of the path from $\digr_{\sigma_1}$ to
$\digr_{\sigma_2}$ if necessary, and we obtain a path
$P'\in\Pi_{ik',t'}$ where $t'\equiv l(\modd\;\gamma^{\ulabel})$. By
Lemma~\ref{l:paths} we obtain that the growth rate of
$\{a_{ik'}^{l+s\gamma},\; s\geq 0\}$ is not less than $\lambda$,
hence it must be $\lambda$.
\end{proof}

To assess the computational feasibility of condition 2. in
Theorem~\ref{t:totalper} we need the following observation which
uses the ultimate expansion, see Theorem~\ref{t:serg-sch} and
Section~\ref{s:ultexp}.

\begin{theorem}
\label{t:cond-feas} Let $A\in\Rpnn$ be such that the sequences
$\{A^te_i,\;t\geq 0\}$ are ultimately linear periodic for all $i\in
L^{\ulabel}$. Then $A$ is orbit periodic if and only if the
following holds for all $\mu,\nu\in\Sigma$:
\begin{equation}
\label{e:cond-feas} \lambda_{\mu}<\lambda_{\nu}\Rightarrow
\bigcup_{i\in N_{\mu}^{\ulabel}}\supp(\ultim_{\mu}^{(1)}
e_i)\subseteq \bigcap_{j\in N_{\nu}^{\ulabel}}
\supp(\ultim_{\nu}^{(1)}e_j).
\end{equation}
\end{theorem}
\begin{proof}
If $\{A^te_i,\;t\geq 0\}$ is ultimately linear periodic and $i\in
L^{\ulabel}$, then
$A^te_i\ulteq\lambda_{\mu}^t\ultim_{\mu}^{(t)}e_i$ for
$\lambda_{\mu}=\lambda(i)$, since $\lambda_{\mu}$ must be the growth
rate of $A^te_i$ and all other terms of the ultimate expansion have
different growth rates. Then also $A^t(e_i\oplus
e_j)\ulteq\lambda_{\mu}^t\ultim_{\mu}^{(t)}e_i\oplus\lambda_{\nu}^t\ultim_{\nu}^{(t)}e_j$
where $\lambda_{\mu}=\lambda(i)$ and $\lambda_{\nu}=\lambda(j)$
(equivalently, $i\in M_{\mu}^{\ulabel}$ and $j\in
M_{\nu}^{\ulabel}$). If $\lambda_{\mu}<\lambda_{\nu}$ then
$\{A^t(e_i\oplus e_j),\; t\geq 0\}$ is ultimately linear periodic if
and only if $A^t(e_i\oplus e_j)\ulteq \lambda_{\nu}^t\ultim_{\nu}^{(t)}
e_j$. This happens if and only if
\begin{equation}
\label{ultcond} \lambda_{\mu}<\lambda_{\nu}\Rightarrow
\supp(\ultim_{\mu}^{(t)}e_i)\subseteq \supp(\ultim_{\nu}^{(t)}e_j)\
\forall i\in M_{\mu}^{\ulabel},\ j\in M_{\nu}^{\ulabel},
\end{equation}
holds for all $\mu,\nu\in\Sigma$.

Corollary~\ref{c:lindep} implies that any column of
$\ultim_{\mu}^{(t)}$, resp. $\ultim_{\nu}^{(t)}$ is a max-linear
combination of columns with indices in $N_{\mu}^{\ulabel}$, resp.
$N_{\nu}^{\ulabel}$, so that the support of that column is the union
of supports of certain columns with indices in $N_{\mu}^{\ulabel}$,
resp. $N_{\nu}^{\ulabel}$. Hence we need to check the support
inclusions only for $i\in N_{\mu}^{\ulabel}$ and $j\in
N_{\nu}^{\ulabel}$. Further, Corollary~\ref{c:per} implies that the
columns of $\ultim_{\mu}^{(t)}$ (or resp. $\ultim_{\nu}^{(t)}$) with
indices in $N_{\mu}^{\ulabel}$ (or resp. $N_{\nu}^{\ulabel}$) just
permute as $t$ changes, so the inclusions need be verified only for
$t=1$. This shows that \eqref{ultcond} is equivalent to
\begin{equation}
\label{ultcond2} \lambda_{\mu}<\lambda_{\nu}\Rightarrow
\supp(\ultim_{\mu}^{(1)}e_i)\subseteq \supp(\ultim_{\nu}^{(1)}e_j)\
\forall i\in N_{\mu}^{\ulabel},\ j\in N_{\nu}^{\ulabel},
\end{equation}
which is equivalent to \eqref{e:cond-feas}. The claim is proved.
\end{proof}

Now we give a polynomial bound on the computational complexity of
verifying the orbit periodicity of a reducible matrix.

\begin{theorem}
\label{t:compcomp-totper} Let $A\in\Rpnn$. Suppose that all
components of $\digr(A)$ and access relations between them are
known, and $\lambda_{\mu}$ and $\ultim_{\mu}^{(1)}$ are computed for
all $\mu\in\Sigma$. Then the orbit periodicity of $A$ can be
verified in no more than $O(n^3)$ operations.
\end{theorem}
\begin{proof}
For all pairs $\mu,\nu$, we must verify condition 1. of
Theorem~\ref{t:totalper} and condition \eqref{e:cond-feas}. The
first of these conditions is verified for the pairs of components of
$\digr(A)$ and it takes no more than $O(n^2)$, if all access
relations between them are known. To verify condition
\eqref{e:cond-feas}, we need to compute all unions $\bigcup_{i\in
N_{\mu}^{\ulabel}}\supp(\ultim_{\mu}^{(1)} e_i)$ and intersections
$\bigcap_{i\in N_{\mu}^{\ulabel}}\supp(\ultim_{\mu}^{(1)} e_i)$,
which requires $O(n^2)$ operations, and then make $O(n^2)$
comparisons of Boolean vectors, which requires no more than $O(n^3)$
operations.
\end{proof}

We combine the results of Theorems~\ref{t:compcomp}
and~\ref{t:compcomp-totper}.

\begin{corollary}
\label{c:compcomp} Given $A\in\Rpnn$, it takes no more than
$O(n^4\log n)$ operations to verify whether it is orbit periodic or
not.
\end{corollary}

\section{Examples}
\label{s:ex}

All examples in this section will be in the max-plus setting
$\R_{\max,+}:=(\R\cup\{-\infty\},\oplus=\max,\otimes=+)$.

{\em Example 1.} We construct the canonical Nachtigall expansion of
$A^t$ for
\begin{equation}
A=
\begin{pmatrix}
-1 & 0 & -7 & -6\\
0 & -1 & -5 & -4\\
-7 & -5 & -1 & -3\\
-6 & -4 & -3 & -2
\end{pmatrix}.
\end{equation}
We start with $A_1=A$. The maximal cycle mean is $\lambda_1=0$, and
the component $\crit_1$ has two nodes $1,2$ and two edges $(1,2)$
and $(2,1)$. The cyclicity is $\gamma_1:=2$.

We proceed by setting the entries in first two rows and columns to
$-\infty$, thus obtaining
\begin{equation}
A_2=
\begin{pmatrix}
-\infty & -\infty & -\infty & -\infty\\
-\infty & -\infty & -\infty & -\infty\\
-\infty & -\infty & -1 & -3\\
-\infty & -\infty & -3 & -2
\end{pmatrix}.
\end{equation}
The maximum cycle mean is $\lambda_2=-1$, and the component
$\crit_2$ has one node $3$ and one edge $(3,3)$. The cyclicity is
$\gamma_2=1$.

Now we set everything to $-\infty$ except for the entry $(4,4)$:
\begin{equation}
A_3=
\begin{pmatrix}
-\infty & -\infty & -\infty & -\infty\\
-\infty & -\infty & -\infty & -\infty\\
-\infty & -\infty & -\infty & -\infty\\
-\infty & -\infty & -\infty & -2
\end{pmatrix}.
\end{equation}
The maximum cycle mean is $\lambda_3=-2$, and the component
$\crit_3$ has one node $4$ and one edge $(4,4)$. The cyclicity is
$\gamma_3=1$.

We obtain matrices $S_1$, $S_2$ and $S_3$ which correspond,
respectively, to $\crit_1$, $\crit_2$ and $\crit_3$:
\begin{equation}
\label{s1s2s3}
\begin{split}
S_1&=
\begin{pmatrix}
-\infty & 0 & -\infty & -\infty\\
0 & -\infty & -\infty & -\infty\\
-\infty & -\infty & -\infty & -\infty\\
-\infty & -\infty & -\infty & -\infty\\
\end{pmatrix},\quad
S_2=
\begin{pmatrix}
-\infty & -\infty & -\infty & -\infty\\
-\infty & -\infty & -\infty & -\infty\\
-\infty & -\infty & 0 & -\infty\\
-\infty & -\infty & -\infty & -\infty\\
\end{pmatrix},\\
S_3&=
\begin{pmatrix}
-\infty & -\infty & -\infty & -\infty\\
-\infty & -\infty & -\infty & -\infty\\
-\infty & -\infty & -\infty & -\infty\\
-\infty & -\infty & -\infty & 0\\
\end{pmatrix}.
\end{split}
\end{equation}
These are Boolean matrices in the max-plus setting, with entries
$\infty,0$ instead of $0,1$.

Further we need to compute Kleene stars $(A^{\gamma_1})^*$,
$((A_2-\lambda_2)^{\gamma_2})^*=(1+A_2)^*$ and
$((A_3-\lambda_3)^{\gamma_3})^*=(2+A_3)^*$, and construct matrices
$C_1$, $R_1$, $C_2$, $R_2$, $C_3$ and $R_3$. The critical parts of
the Kleene stars are shown below, the rest of the elements being
denoted by $\cdot$ as we do not need them:
\begin{equation}
\begin{split}
(A^2)^*&=
\begin{pmatrix}
0 & -1 & -5 & -4\\
-1 & 0 & -6 & -5\\
-5 & -6 & \cdot & \cdot\\
-4 & -5 & \cdot & \cdot
\end{pmatrix},\\
(1\otimes A_2)^*&=
\begin{pmatrix}
\cdot &\cdot & -\infty &\cdot\\
\cdot &\cdot & -\infty &\cdot\\
-\infty & -\infty & 0 & -2\\
\cdot & \cdot & -2 &\cdot
\end{pmatrix}, \quad
(2\otimes A_3)^*=
\begin{pmatrix}
\cdot &\cdot & \cdot & -\infty\\
\cdot &\cdot & \cdot & -\infty\\
\cdot & \cdot & \cdot & -\infty\\
-\infty & -\infty & -\infty & 0
\end{pmatrix}.
\end{split}
\end{equation}
Further we compute the Nachtigall matrices
\begin{equation}
\label{e:nachtmat}
\begin{split}
\nacht_1^{(0)}&=C_1\otimes R_1=
\begin{pmatrix}
0 & -1 & -5 & -4\\
-1 & 0 & -6 & -5\\
-5 & -6 & -10 & -9\\
-4 & -5 & -9 & -8
\end{pmatrix},\\
\nacht_1^{(1)}&=C_1\otimes S_1\otimes R_1=
\begin{pmatrix}
-1 & 0 & -6 & -5\\
0 & -1 & -5 & -4\\
-6 & -5 & -11 & -10\\
-5 & -4 & -10 & -9
\end{pmatrix},\\
\nacht_2^{(0)}&=C_2\otimes R_2=
\begin{pmatrix}
-\infty & -\infty & -\infty & -\infty\\
 -\infty & -\infty & -\infty & -\infty\\
-\infty & -\infty & 0 & -2\\
-\infty & -\infty & -2 & -4
\end{pmatrix},\\
\nacht_3^{(0)}&=C_3\otimes R_3=
\begin{pmatrix}
-\infty & -\infty & -\infty & -\infty\\
 -\infty & -\infty & -\infty & -\infty\\
-\infty & -\infty & -\infty & -\infty\\
-\infty & -\infty & -\infty & 0
\end{pmatrix}.
\end{split}
\end{equation}
The Nachtigall expansion starts to work already at $t=2$. Indeed,
\begin{equation}
\begin{split}
A^2&=
\begin{pmatrix}
0 & -1 & -5 & -4\\
-1 & 0 & -6 & -5\\
-5 & -6 & -2 & -4\\
-4 & -5 & -4 & -4
\end{pmatrix}
=\nacht_1^{(0)}\oplus (-2)\otimes\nacht_2^{(0)}\oplus (-4)\otimes\nacht_3^{(0)},\\
A^3&=
\begin{pmatrix}
-1 & 0 & -6 & -5\\
0 & -1 & -5 & -4\\
-6 & -5 & -3 & -5\\
-5 & -4 & -5 & -6
\end{pmatrix}
=\nacht_1^{(1)}\oplus(-3)\otimes\nacht_2^{(0)}\oplus
(-6)\otimes\nacht_3^{(0)}.
\end{split}
\end{equation}

Starting from $t=4$ the third term can be forgotten:
\begin{equation}
A^4=
\begin{pmatrix}
0 &-1 & -5 & -4\\
-1 & 0 & -6 & -5\\
-5 & -6 & -4 & -6\\
-4 & -5 & -6 & -8
\end{pmatrix}
=\nacht_1^{(0)}\oplus (-4)\otimes\nacht_2^{(0)}.
\end{equation}
The ultimate periodic behavior starts after $T(A)=10$:
\begin{equation}
A^{10}=
\begin{pmatrix}
0 & -1 & -5 & -4\\
-1 & 0 & -6 & -5\\
-5 & -6 & -10 & -9\\
-4 & -5 & -9 & -8
\end{pmatrix}
=\nacht_1^{(0)}=\ultim_1^{(0)}.
\end{equation}

{\em Example 2.}  The following example will illustrate the ultimate
expansion:
\begin{equation}
A=
\begin{pmatrix}
-2 & 0 & -3 & -7 & -\infty & -\infty & -\infty\\
0 & -2 & -5 & -7 & -\infty & -\infty & -\infty\\
-9 & -7 & -9 & -8 & -\infty & -\infty & -\infty\\
-9 & -6 & -4 & -4 & -\infty & -\infty  & -\infty\\
-8 & -5 & -5 & -4 & -1 & -7 & -5\\
-7 & -8 & -5 & -6 & -3 & -6 & -8\\
-6 & -4 & -9 & -3 & -5 & -5 & -5
\end{pmatrix}.
\end{equation}

We compute the elements of the ultimate expansion.

Firstly, $A_1^{\ulabel}=A$, $\lambda_1=0$, and the component
$\crit_1^{\ulabel}=\crit(A)$ consists of two nodes $1,2$ and two
edges $(1,2)$ and $(2,1)$.

On the next step we set all entries in the first four rows and
columns of $A$ to $-\infty$, thus obtaining matrix $A_2^{\ulabel}$.
Its {\em essential} submatrix with finite entries is extracted from
the remaining rows and columns 5 to 7:
\begin{equation}
\label{a2-ess} A_{2\;ess}^{\ulabel}=
\begin{pmatrix}
 -1 & -7 & -5\\
 -3 & -6 & -8\\
 -5 & -5 & -5
\end{pmatrix},\quad
\text{5 to 7 $\times$ 5 to 7}.
\end{equation}
We compute $\lambda_2=-1$, and the component $\crit_2=\crit(A_2)$
consists of the loop $(5,5)$.

The components $\crit_1$ and $\crit_2$ determine Boolean (i.e.,
$0,-\infty$) matrices $S_1$ and $S_2$. We compute $(A^2)^*$ and
$(1\otimes A_2^{\ulabel}) ^*$:
\begin{equation}
(A^2)^*=
\begin{pmatrix}
0 & -2 & -5 & -7 & -\infty & -\infty & -\infty\\
-2 & 0 & -3 & -7 & -\infty & -\infty & -\infty\\
-7 & -9 & 0 & -12 & -\infty & -\infty & -\infty\\
-6 & -8 & -8 & 0 & -\infty & -\infty  & -\infty\\
-5 & -6 & -6 & -5 & 0 & -8 & -6\\
-8 & -7 & -8 & -7 & -4 & 0 & -8\\
-4 & -6 & -7 & -7 & -6 & -10 & 0
\end{pmatrix},
\end{equation}
\begin{equation}
(1\otimes A_{2\;ess}^{\ulabel})^*=
\begin{pmatrix}
 0 & -6 & -4\\
 -2 & 0 & -6\\
 -4 & -4 & 0
\end{pmatrix},\quad
\text{5 to 7 $\times$ 5 to 7}.
\end{equation}

Next we build matrices $C_1^{\ulabel}, R_1^{\ulabel}$ and
$C_2^{\ulabel}, R_2^{\ulabel}$ whose essential parts are shown
below:

\begin{equation}
\label{c1r1}
\begin{split}
C_{1\; ess}^{\ulabel}&=
\begin{pmatrix}
0 & -2 & -7 & -6 & -5 & -8 & -4\\
-2 & 0 & -9 & -8 & -6 & -7 & -6
\end{pmatrix}^T,\ \text{1 to 7$\times$ 1 to 2,}\\
R_{1\; ess}^{\ulabel}&=
\begin{pmatrix}
0 & -2 & -5 & -7 \\
-2 & 0 & -3 & -7
\end{pmatrix},\ \text{1 to 2$\times$ 1 to 4.}\\
\end{split}
\end{equation}

\begin{equation}
\label{c2r2}
\begin{split}
C_{2\; ess}^{\ulabel}&=
\begin{pmatrix}
 0 & -2 & -4
\end{pmatrix}^T,\quad \text{5 to 7$\times$ 5}\\
R_{2\; ess}^{\ulabel}&=
\begin{pmatrix}
0 & -6 & -4
\end{pmatrix},\quad \text{5 $\times$ 5 to 7}.
\end{split}
\end{equation}
Using \eqref{c1r1} and \eqref{c2r2} we compute the ultimate terms
$\ultim_1^{(0)}=C_1^{\ulabel}\otimes R_1^{\ulabel}$,
$\ultim_1^{(1)}=C_1^{\ulabel}\otimes S_1^{\ulabel}\otimes
R_1^{\ulabel}$ and $\ultim_2^{(0)}=C_2^{\ulabel}\otimes
R_2^{\ulabel}$. The ultimate expansion starts to work at $t=9$,
meaning
\begin{equation}
\label{ult-work}
\begin{split}
A^9&=\ultim_1^{(1)}\oplus(-9)\otimes\ultim_2^{(0)},\
A^{10}=\ultim_1^{(0)}\oplus (-10)\otimes\ultim_2^{(0)},\\
A^{11}&=\ultim_1^{(1)}\oplus(-11)\otimes\ultim_2^{(0)},\
A^{12}=\ultim_1^{(0)}\oplus (-12)\otimes\ultim_2^{(0)},\;\ldots
\end{split}
\end{equation}

In this case, the canonical Nachtigall expansion starts to work
already at $t=3$, and after $t=4$ only the first two terms are
essential. In this expansion, matrix $A_2$ results from setting only
the first {\em two} (instead of four) columns and rows to $-\infty$.
The second Nachtigall term $\nacht_2$ is equal to $C_2\otimes R_2$,
where $C_2=C_2^{\ulabel}$ and $R_2\neq R_2^{\ulabel}$ has essential
part
\begin{equation}
R_{2\;ess}=
\begin{pmatrix}
-4 & -3 & 0 & -6 & -4
\end{pmatrix},\quad\text{5 $\times$ 3 to 7}.
\end{equation}

{\em Example 3.} We illustrate the total periodicity. Let
\begin{equation}
A=
\begin{pmatrix}
-\infty & 1 & -\infty & -\infty & -\infty & -\infty\\
-\infty & -\infty & 1 & -\infty & -\infty & -\infty\\
-\infty & -\infty & -\infty & 1 & -\infty & -\infty\\
1 & -\infty & -\infty & -\infty & -\infty & -\infty\\
-\infty & -2 & -\infty & -\infty & -\infty & 0\\
-\infty & -2 & -\infty & -\infty & 0 & -\infty
\end{pmatrix}
\end{equation}
\begin{equation}
B=
\begin{pmatrix}
-\infty & 1 & -\infty & -\infty & -\infty & -\infty\\
-\infty & -\infty & 1 & -\infty & -\infty & -\infty\\
-\infty & -\infty & -\infty & 1 & -\infty & -\infty\\
1 & -\infty & -\infty & -\infty & -\infty & -\infty\\
-\infty & -2 & -\infty & -\infty & -\infty & 0\\
-\infty & -\infty & -2 & -\infty & 0 & -\infty
\end{pmatrix}
\end{equation}
Note that $A$ and $B$ are almost the same, except for the entries
$(6,2)$ and $(6,3)$. In both cases the ultimate expansion coincides
with the canonical Nachtigall expansion, and we have two critical
components $\crit_1=(N_1,E_1)$ with $N_1=\{1,2,3,4\}$,
$E_1=\{(1,2),(2,3),(3,4),(4,1)\}$ and $\crit_2=(N_2,E_2)$ with
$N_2=\{5,6\}$ and $E_2=\{(5,6),(6,5)\}$. The eigenvalues are
$\lambda_1=1$ and $\lambda_2=0$, and the cyclicities are
$\gamma_1=4$ and $\gamma_2=2$, their l.c.m. is $\gamma=4$. Notr that
$M_1^{\ulabel}=N_1$ and $M_2^{\ulabel}=N_2$.

In both cases the component with $\lambda_1$ does not have access to
the component with $\lambda_2$ which is smaller, hence condition 1.
of Theorem \ref{t:totalper} is true meaning that all columns of
$A^t$ and $B^t$ are ultimately periodic. Condition 2. of Theorem
\ref{t:totalper} holds for $A$ but it does not hold for $B$. In
particular, there are only paths of odd length connecting node $6$
to node $3$. Hence $A$ is orbit periodic and $B$ is not.

Consider also condition \eqref{e:cond-feas}. We need the terms of
the ultimate expansion for $t(\modd\;\gamma^{\ulabel})=1$. In each
case there are two terms, which we denote by $\ultim_1^A$ and
$\ultim_2^A$, resp. $\ultim_1^B$ and $\ultim_2^B$, for the case of
$A$, resp. $B$. We have
\begin{equation}
\ultim_1^A=
\begin{pmatrix}
-\infty & 0 & -\infty & -\infty & -\infty & -\infty\\
-\infty & -\infty & 0 & -\infty & -\infty & -\infty\\
-\infty & -\infty & -\infty & 0 & -\infty & -\infty\\
0 & -\infty & -\infty & -\infty & -\infty & -\infty\\
-4 & -3 & -6 & -5 & -\infty & -\infty\\
-4 & -3 & -6 & -5 & -\infty & -\infty
\end{pmatrix}
\end{equation}
\begin{equation}
\ultim_2^B=
\begin{pmatrix}
-\infty & 0 & -\infty & -\infty & -\infty & -\infty\\
-\infty & -\infty & 0 & -\infty & -\infty & -\infty\\
-\infty & -\infty & -\infty & 0 & -\infty & -\infty\\
0 & -\infty & -\infty & -\infty & -\infty & -\infty\\
-\infty & -3 & -\infty & -5 & -\infty & -\infty\\
-4 & -\infty & -3 & -\infty & -\infty & -\infty
\end{pmatrix}
\end{equation}
\begin{equation}
\ultim_{2\;ess}^A=\ultim_{2\;ess}^B=
\begin{pmatrix}
-\infty & 0\\
0 & -\infty
\end{pmatrix},\quad\text{5 to 6 $\times$ 5 to 6.}
\end{equation}
Observe that $\supp(\ultim_2^Ae_i)\subseteq\supp(\ultim_1^Ae_j)$ for
all $i\in M_2^{\ulabel}=\{5,6\}$ and $j\in
M_1^{\ulabel}=\{1,2,3,4\}$, but this condition does not hold for
$B$.

To see the difference between $A^tx$ and $B^tx$, take
$x=[0\;-\infty\;-\infty\;-\infty\;-\infty\; 0]^T$. The sequence
$\{A^tx\}$ is ultimately periodic starting from $t=4$ with period
$4$ and growth rate $1$. In particular the last component of
$\{A^tx\}$ yields the following number sequence for $t\geq 4$:
\begin{equation}
\label{atxtgeq4}
(A^tx)_6=\{1\;1\;1\;1\;5\;5\;5\;5\;9\;9\;9\;9\ldots\},\ t\geq 4.
\end{equation}
The sequence $\{B^tx\}$ is not ultimately periodic. In particular,
the last component of $\{B^tx\}$ yields the following number
sequence for $t\geq 2$:
\begin{equation}
\label{atxtgeq41}
(B^tx)_6=\{0\;0\;0\;1\;0\;4\;0\;5\;0\;8\;0\;9\ldots\},\ t\geq 2,
\end{equation}
which can be expressed as
\begin{equation}
(B^tx)_6=
\begin{cases}
0, &\text{if $t$ is even and $t\geq 2$,}\\
t-3, &\text{if $t=4k+3$ and $k\geq 0$,}\\
t-5, &\text{if $t=4k+5$ and $k\geq 0$.}
\end{cases}
\end{equation}

\section{Acknowledgement}
We thank Peter Butkovi\v{c} for his encouraging support, numerous discussions
and careful reading of the paper. We acknowledge the work of the anonimous reviewer, who helped us to eliminate some mistakes and typos. We are also grateful to Abdulhadi Aminu,
Trivikram Dokka and Glenn Merlet for their comments, help, and advice.


\providecommand{\bysame}{\leavevmode\hbox to3em{\hrulefill}\thinspace}
\providecommand{\MR}{\relax\ifhmode\unskip\space\fi MR }
\providecommand{\MRhref}[2]{%
  \href{http://www.ams.org/mathscinet-getitem?mr=#1}{#2}
}
\providecommand{\href}[2]{#2}

\end{document}